\documentclass[11pt,oneside]{amsart}
\usepackage[latin9]{inputenc}
\usepackage{amsthm}
\usepackage{amssymb}
\PassOptionsToPackage{normalem}{ulem}
\usepackage{ulem}

\makeatletter
\numberwithin{equation}{section}
\numberwithin{figure}{section}
\theoremstyle{plain}
\newtheorem{thm}{\protect\theoremname}[section]
  \theoremstyle{definition}
  \newtheorem{defn}[thm]{\protect\definitionname}
  \theoremstyle{plain}
  \newtheorem{lem}[thm]{\protect\lemmaname}
  \theoremstyle{plain}
  \newtheorem{cor}[thm]{\protect\corollaryname}
  \theoremstyle{definition}
  \newtheorem{example}[thm]{\protect\examplename}
  \theoremstyle{plain}
  \newtheorem{fact}[thm]{\protect\factname}
  \theoremstyle{plain}
  \newtheorem{conjecture}[thm]{\protect\conjecturename}
  \theoremstyle{remark}
  \newtheorem{rem}[thm]{\protect\remarkname}
  \theoremstyle{plain}
  \newtheorem{prop}[thm]{\protect\propositionname}
  \theoremstyle{remark}
  \newtheorem{claim}[thm]{\protect\claimname}
  \theoremstyle{remark}
  \newtheorem*{claim*}{\protect\claimname}
  \theoremstyle{definition}
  \newtheorem{problem}[thm]{\protect\problemname}

\usepackage{amsfonts}
\usepackage{nicefrac}
\usepackage{amscd}
\usepackage{a4wide}
\usepackage{url}

\makeatother

  \providecommand{\claimname}{Claim}
  \providecommand{\conjecturename}{Conjecture}
  \providecommand{\corollaryname}{Corollary}
  \providecommand{\definitionname}{Definition}
  \providecommand{\examplename}{Example}
  \providecommand{\factname}{Fact}
  \providecommand{\lemmaname}{Lemma}
  \providecommand{\problemname}{Problem}
  \providecommand{\propositionname}{Proposition}
  \providecommand{\remarkname}{Remark}
\providecommand{\theoremname}{Theorem}

\begin{document}
\global\long\def\p{\mathbf{p}}
\global\long\def\q{\mathbf{q}}
\global\long\def\C{\mathfrak{C}}
\global\long\def\SS{\mathcal{P}}
 \global\long\def\pr{\operatorname{pr}}
\global\long\def\image{\operatorname{im}}
\global\long\def\otp{\operatorname{otp}}
\global\long\def\dec{\operatorname{dec}}
\global\long\def\suc{\operatorname{suc}}
\global\long\def\pre{\operatorname{pre}}
\global\long\def\qe{\operatorname{qf}}
 \global\long\def\ind{\operatorname{ind}}
\global\long\def\Nind{\operatorname{Nind}}
\global\long\def\lev{\operatorname{lev}}
\global\long\def\Suc{\operatorname{Suc}}
\global\long\def\HNind{\operatorname{HNind}}
\global\long\def\minb{{\lim}}
\global\long\def\concat{\frown}
\global\long\def\cl{\operatorname{cl}}
\global\long\def\tp{\operatorname{tp}}
\global\long\def\id{\operatorname{id}}
\global\long\def\cons{\left(\star\right)}
\global\long\def\qf{\operatorname{qf}}
\global\long\def\ai{\operatorname{ai}}
\global\long\def\dtp{\operatorname{dtp}}
\global\long\def\acl{\operatorname{acl}}
\global\long\def\nb{\operatorname{nb}}
\global\long\def\limb{{\lim}}
\global\long\def\leftexp#1#2{{\vphantom{#2}}^{#1}{#2}}
\global\long\def\intr{\operatorname{interval}}
\global\long\def\atom{\emph{at}}
\global\long\def\I{\mathfrak{I}}
\global\long\def\uf{\operatorname{uf}}
\global\long\def\ded{\operatorname{ded}}
\global\long\def\Ded{\operatorname{Ded}}
\global\long\def\Df{\operatorname{Df}}
\global\long\def\Th{\operatorname{Th}}
\global\long\def\eq{\operatorname{eq}}
\global\long\def\Aut{\operatorname{Aut}}
\global\long\def\ac{ac}
\global\long\def\DfOne{\operatorname{df}_{\operatorname{iso}}}
\global\long\def\modp#1{\pmod#1}
\global\long\def\sequence#1#2{\left\langle #1\left|\,#2\right.\right\rangle }
\global\long\def\set#1#2{\left\{  #1\left|\,#2\right.\right\}  }
\global\long\def\Diag{\operatorname{Diag}}
\global\long\def\Nn{\mathbb{N}}
\global\long\def\mathrela#1{\mathrel{#1}}
\global\long\def\twiddle{\mathord{\sim}}
\global\long\def\mathordi#1{\mathord{#1}}
\global\long\def\Qq{\mathbb{Q}}
\global\long\def\dense{\operatorname{dense}}
 \global\long\def\cof{\operatorname{cof}}
\global\long\def\tr{\operatorname{tr}}
\global\long\def\treeexp#1#2{#1^{\left\langle #2\right\rangle _{\tr}}}
\global\long\def\x{\times}
\global\long\def\forces{\Vdash}
\global\long\def\Vv{\mathbb{V}}
\global\long\def\Uu{\mathbb{U}}
\global\long\def\tauname{\dot{\tau}}
\global\long\def\ScottPsi{\Psi}
\global\long\def\cont{2^{\aleph_{0}}}
\global\long\def\MA#1{{MA}_{#1}}
\global\long\def\rank#1#2{R_{#1}\left(#2\right)}
\global\long\def\cal#1{\mathcal{#1}}

\def\Ind#1#2{#1\setbox0=\hbox{$#1x$}\kern\wd0\hbox to 0pt{\hss$#1\mid$\hss} \lower.9\ht0\hbox to 0pt{\hss$#1\smile$\hss}\kern\wd0} 
\def\Notind#1#2{#1\setbox0=\hbox{$#1x$}\kern\wd0\hbox to 0pt{\mathchardef \nn="3236\hss$#1\nn$\kern1.4\wd0\hss}\hbox to 0pt{\hss$#1\mid$\hss}\lower.9\ht0 \hbox to 0pt{\hss$#1\smile$\hss}\kern\wd0} 
\def\nind{\mathop{\mathpalette\Notind{}}}

\global\long\def\ind{\mathop{\mathpalette\Ind{}}}
 \global\long\def\nind{\mathop{\mathpalette\Notind{}}}
\global\long\def\average#1#2#3{Av_{#3}\left(#1/#2\right)}
\global\long\def\Ff{\mathfrak{F}}
\global\long\def\mx#1{Mx_{#1}}
\global\long\def\maps{\mathfrak{L}}

\global\long\def\Esat{E_{\mbox{sat}}}
\global\long\def\Ebnf{E_{\mbox{rep}}}
\global\long\def\Ecom{E_{\mbox{com}}}

\global\long\def\init{\trianglelefteq}
\global\long\def\fini{\trianglerighteq}
\global\long\def\K{{\bf K}_{\lambda,\theta}^{M,C,\bar{d}}}
\global\long\def\mxK{{\bf MxK}_{\lambda,\theta}^{M,C,\bar{d}}}
\global\long\def\sch#1#2{\Phi_{#1,#2}}
\global\long\def\xx{{\bf x}}
\global\long\def\yy{{\bf y}}
\global\long\def\zz{{\bf z}}
\global\long\def\aa{\mathfrak{a}}
\global\long\def\bb{\mathfrak{b}}

\title{THE GENERIC PAIR CONJECTURE FOR DEPENDENT FINITE DIAGRAMS}

\author{Itay Kaplan, Noa Lavi, Saharon Shelah}

\thanks{The first author would like to thank the Israel Science Foundation
for partial support of this research (Grant no. 1533/14). }

\thanks{The research leading to these results has received funding from the
European Research Council, ERC Grant Agreement n. 338821. No. 1055
on the third author's list of publications.}

\address{Itay Kaplan \\
The Hebrew University of Jerusalem\\
Einstein Institute of Mathematics \\
Edmond J. Safra Campus, Givat Ram\\
Jerusalem 91904, Israel}

\email{kaplan@math.huji.ac.il}

\urladdr{https://sites.google.com/site/itay80/ }

\address{Noa Lavi\\
The Hebrew University of Jerusalem\\
Einstein Institute of Mathematics \\
Edmond J. Safra Campus, Givat Ram\\
Jerusalem 91904, Israel}

\email{noa.lavi@mail.huji.ac.il}

\address{Saharon Shelah\\
The Hebrew University of Jerusalem\\
Einstein Institute of Mathematics \\
Edmond J. Safra Campus, Givat Ram\\
Jerusalem 91904, Israel}

\address{Saharon Shelah \\
Department of Mathematics\\
Hill Center-Busch Campus\\
Rutgers, The State University of New Jersey\\
110 Frelinghuysen Road\\
Piscataway, NJ 08854-8019 USA}

\email{shelah@math.huji.ac.il}

\urladdr{http://shelah.logic.at/}

\begin{abstract}
This paper generalizes Shelah's generic pair conjecture (now theorem)
for the measurable cardinal case from first order theories to finite
diagrams. We use homogeneous models in the place of saturated models. 
\end{abstract}

\maketitle

\section{Introduction}

The generic pair conjecture states that for every cardinal $\lambda$
such that $\lambda^{+}=2^{\lambda}$ and $\lambda^{<\lambda}=\lambda$,
a complete first order theory $T$ is dependent if and only if, whenever
$M$ is a saturated model whose size is $\lambda^{+}$, then, after
writing $M=\bigcup_{\alpha<\lambda^{+}}M_{\alpha}$ where $M_{\alpha}$
are models of size $\lambda$, there is a club of $\lambda^{+}$ such
that for every pair of ordinals $\alpha<\beta$ of cofinality $\lambda$
from the club, the pair of models $\left(M_{\beta},M_{\alpha}\right)$
has  the same isomorphism type. 

This conjecture is now proved for $\lambda$ large enough. The non-structure
side is proved in \cite{Sh877,Sh906} and the other direction is proved
in \cite{Sh:900,Sh950}, all by the third author. In \cite{Sh:900},
the theorem is proved for the case where $\lambda$ is measurable.
This is the easiest case of the theorem, and this is the case we will
focus on here. In \cite[Theorem 7.3]{Sh950}, the conjecture is proved
when $\lambda>\left|T\right|^{+}+\beth_{\omega}^{+}$. 

The current paper has two agendas. 

The first is to serve as an exposition for the proof of the theorem
in the case where $\lambda$ is measurable. There are already two
expositions by Pierre Simon on some other parts from \cite{Sh:900,Sh950},
which are available on his website%
\footnote{\url{http://www.normalesup.org/~simon/notes.html}%
}. 

The second is to generalize the structure side of this theorem in
the measurable cardinal case to finite diagrams. As an easy byproduct,
we also generalize a weak version of the ``recounting of types''
result \cite[Conclusion 3.13]{Sh950}, which states that when $\lambda$
is measurable and $M$ is saturated of cardinality $\lambda$, then
the number of types over $M$ up to conjugation is $\leq\lambda$.
See Corollary \ref{cor:main} below. 

A finite diagram $D$ is a collection of types in finitely many variables
over $\emptyset$ in some complete theory $T$. Once we fix such a
$D$ we concentrate on $D$-models, which are models of $T$ which
realize only types from $D$. For instance, in a theory with infinitely
many unary predicates $P_{i}$, $D$ could prohibit $x\notin P_{i}$
for all $i$, thus $D$-models are just union of the $P_{i}$'s. In
this context, saturated models become $D$-saturated models, which
is the same as being homogenous and realize $D$ (see Lemma \ref{lem:Grossberg}),
so our model $M$ will be $D$-saturated instead of saturated. 

We propose a definition for when a finite diagram $D$ is dependent.
This definition has the feature that if the underlying theory is dependent,
then so is $D$, so there are many examples of such diagrams. We also
give an example of an independent theory $T$ with some dependent
$D$ (Example \ref{exa:independent theory, dependent diagram}). 

The proof follows \cite{Sh:900} and also uses constructions from
\cite{Sh950}%
\footnote{Instead of ``strict decompositions'' from \cite{Sh:900} we use
$tK$ from \cite{Sh950}.%
}. However, In order to make the proof work, we will need the presence
of a strongly compact cardinal $\theta$ that will help us ensure
that the types we get are $D$-types and so realized in the $D$-saturated
models.

\subsubsection*{Organization of the paper}

In Section \ref{sec:preliminaries} we expose finite diagrams and
prove or cite all the facts we shall need about them and about measurable
and strongly compact cardinals. We also give a precise definition
of when a diagram $D$ is dependent, and prove several equivalent
formulations.

In Section \ref{sec:The-generic-pair} we state the generic pair conjecture
in the terminology of finite diagrams, and give a general framework
for proving it: we introduce \emph{decompositions} and \emph{good
families} and prove that if such things exist, then the theorem is
true. 

Section \ref{sec:The-type-decomposition} is devoted to proving that
nice decompositions exist. This is done in two steps. In Section \ref{sub:Tree-type-decomposition}
we construct the first kind of decomposition (\emph{tree-type decomposition}),
which is the building block of the decomposition constructed in Section
\ref{sub:Self-solvable-decomposition} (\emph{self-solvable decomposition}). 

In Section \ref{sec:Finding-a-good} we prove that the family of self-solvable
decompositions over a $D$-saturated model form a good family, and
deduce the generic pair conjecture.

\paragraph*{Acknowledgment}

We would like to thank the anonymous referee for his careful reading,
for his many useful comments and for finding several inaccuracies.

\section{\label{sec:preliminaries}preliminaries}

We start by giving the definition of homogeneous structures and of
$D$-models.
\begin{defn}
Let $M$ be some structure in some language $L$. We say that $M$
is\emph{ $\kappa$-homogeneous}%
\footnote{In some publications this notion is called $\kappa$-sequence homogenous,
but here we decided upon this simpler notation which is also standard,
see\emph{ \cite[page 480, 1.3]{Hod}.}%
} if:
\begin{itemize}
\item for every $A\subseteq M$ with $|A|<\kappa$, every partial elementary
map $f$ defined on $A$ and $a\in M$ there is some $b\in M$ such
that $f\cup\{(a,b)\}$ is an elementary map. 
\end{itemize}
We say that $M$ is \emph{homogeneous} if it is $\left|M\right|$-homogeneous. 
\end{defn}
Note that when $M$ is homogenous, it is also \emph{strongly homogeneous},
meaning that if $f$ is a partial elementary map with domain $A$
such that $\left|A\right|<\left|M\right|$, $f$ extends to an automorphism
of $M$. 

Fix a complete first order theory $T$ in a language $L$ with a monster
model $\C$ --- a saturated model containing all sets and models of
$T$, with cardinality $\bar{\kappa}=\bar{\kappa}^{<\bar{\kappa}}$
bigger than any set or model we will consider. 
\begin{defn}
For $A\subseteq\C$, let $D(A)=\set{\tp(\bar{a}/\emptyset)}{\bar{a}\subseteq A,|\bar{a}|<\omega}$.
A set $D$ of complete $L$-types over $\emptyset$ is a\emph{ finite
diagram in $T$} when it is of the form $D(A)$ for some $A$. If
$D$ is a finite diagram in $L$, then a set $B\subseteq\C$ is a
\emph{$D$-set} if $D\left(B\right)\subseteq D$. A model of $T$
which is a $D$-set is a \emph{$D$-model.}

Let $A\subseteq\C$ be a $D$-set. Let $p$ be a complete type over
$A$ (in any number of variables). We say that $p$ is a \emph{$D$-type}
if for every $\bar{c}$ realizing $p$, $A\cup\bar{c}$ is a $D$-set.
We denote the set of $D$-types over $A$ by $S_{D}\left(A\right)$
(and as usual we use superscript to denote the number of variables,
such as in $S_{D}^{<\omega}\left(A\right)$). We say that $M$ is
\emph{$\left(D,\kappa\right)$-saturated} if whenever $\left|A\right|<\kappa$,
every $p\in S_{D}^{1}\left(A\right)$ is realized in $M$. We say
that $M$ is \emph{$D$-saturated} if it is $\left(D,\left|M\right|\right)$-saturated. 
\end{defn}
Note that when $D$ is\emph{ trivial}, i.e., $D=\bigcup\set{D_{n}\left(T\right)}{n<\omega}$
(with $D_{n}\left(T\right)$ being the set of all complete $n$-types
over $\emptyset$), every model of $T$ is a $D$-model. 

The connection between $D$-saturation and homogeneity becomes clear
due to the following lemma.
\begin{lem}
\label{lem:Grossberg}\cite[Lemma 2.4]{grossbergLessmann} Let $D$
be a finite diagram. A $D$-model $M$ is $\left(D,\kappa\right)$-saturated
if and only if $D\left(M\right)=D$ and $M$ is $\kappa$-homogeneous. 
\end{lem}
Just as in the first order case, we get the following.
\begin{cor}
\label{cor:homo-isom} Let $D$ be a finite diagram. If $M$, $N$
are $D$-saturated of the same cardinality, then $M\cong N$. Furthermore,
if $\lambda^{<\lambda}=\lambda$, and there is a $\left(D,\lambda\right)$-saturated
model, then there exists a $D$-saturated model of size $\lambda$. 
\end{cor}
The next natural thing, after obtaining this equivalence, would be
to look for monsters. A diagram $D$ is \emph{good} if for every $\lambda$
there exists a $(D,\lambda)$-saturated model (see \cite[Definition 2.1]{Sh003}).
We will assume throughout that $D$ is good. By Corollary \ref{cor:homo-isom},
as we assumed that $\bar{\kappa}^{<\bar{\kappa}}=\bar{\kappa}$, there
is a $D$-saturated model $\C_{D}\prec\C$ of cardinality $\bar{\kappa}$
--- the homogenous monster. From now on we make these assumptions
without mentioning them explicitly. 

Let us recall the general notion of an average type along an ultrafilter.
\begin{defn}
\label{def:avgtp} Let $A\subseteq\C_{D}$, $I$ some index set, $\bar{a}_{i}$
tuples of the same length for $i\in I$, and let $\cal U$ be an ultrafilter
on $I$. The \emph{average type} $\average{\sequence{\bar{a_{i}}}{i\in I}}A{\cal U}$
is the type consisting of all the formulas $\phi\left(\bar{x},\bar{c}\right)$
over $A$ such that $\set{i\in I}{\mathfrak{C}_{D}\models\phi\left(\bar{a}_{i},\bar{c}\right)}\in\cal U$. 
\end{defn}
When $\cal U$ is $\kappa$-complete, the average is $<\kappa$ satisfiable
in the sequence $\sequence{\bar{a}_{i}}{i\in I}$ (any $<\kappa$
many formulas are realized in the sequence). It follows that the average
type is a $D$-type (see below). 
\begin{lem}
\label{lem:average type along ultrafilter} Let $A,I$ be as in Definition
\ref{def:avgtp}, and let $\cal U$ be a $\kappa$-complete ultrafilter
on $I$, where $\kappa>\left|T\right|$. Then $r=\average{\sequence{\bar{a_{i}}}{i\in I}}A{\cal U}$
is a $D$-type.\end{lem}
\begin{proof}
We must show that if $\bar{c}\models r$ (in $\C$), then $A\cup\bar{c}$
is a $D$-set. We may assume that $\bar{c}$ is a finite tuple (and
so are the tuples $\bar{a}_{i}$ for $i\in I$). It is enough to see
that if $\bar{c}\bar{a}$ is a finite tuple of elements from $\bar{c}\cup A$,
then for some $i\in I$, $\bar{a}_{i}\bar{a}\equiv\bar{c}\bar{a}$
(i.e., they have the same type over $\emptyset$). For each formula
$\varphi\left(\bar{x},\bar{a}\right)$ such that $\varphi\left(\bar{c},\bar{a}\right)$
holds, the set $\set{i\in I}{\C_{D}\models\varphi\left(\bar{a}_{i},\bar{a}\right)}\in\cal U.$
Since there are $\left|T\right|$ such formulas, by $\kappa$-completeness,
there is some $i\in I$ in the intersection of all these sets, so
we are done. 
\end{proof}
Now we turn to Hanf numbers. Let $\mu\left(\lambda,\kappa\right)$
be the first cardinal $\mu$ such that if $T_{0}$ is a theory of
size $\leq\lambda$, $\Gamma$ a set of finitary types in $T_{0}$
(over $\emptyset$) of cardinality $\leq\kappa$, and for every $\chi<\mu$
there is a model of $T_{0}$ of cardinality $\geq\chi$ omitting all
the types in $\Gamma$, then there is such a model in arbitrarily
large cardinality. Of course, when $\kappa=0$, $\mu\left(\lambda,\kappa\right)=\aleph_{0}$.
In our context, $T_{0}=T$, and $\Gamma=\bigcup\set{D_{n}\left(T\right)}{n<\omega}\backslash D$,
so we are interested in $\mu\left(\left|T\right|,\left|\Gamma\right|\right)$
which we will denote by $\mu\left(D\right)$, the \emph{Hanf number
}of $D$. In \cite[Chapter VII, 5]{Sh:c} this number is given an
upper bound: $\mu\left(D\right)\leq\beth_{\left(2^{\left|T\right|}\right)^{+}}$. 
\begin{defn}
A finite diagram $D$ has \emph{the independence property} if there
exists a formula $\phi\left(\bar{x},\bar{y}\right)$ which has it,
which means that there is an indiscernible sequence $\sequence{\bar{a}_{i}}{i<\mu\left(D\right)}$
and $\bar{b}$ in $\C_{D}$ such that $\C_{D}\models\phi(\bar{b},\bar{a}_{i})$
if and only if $i$ is even. Otherwise we say that $D$ is \emph{dependent}. 
\end{defn}
Of course, if the underlying theory $T$ is dependent, then $D$ is
dependent.
\begin{example}
\label{exa:independent theory, dependent diagram}Let $L=\left\{ R,P,Q\right\} $
where $P$ and $Q$ are unary predicates, and $R$ is a binary predicate.
Let $T$ be the model completion of the theory that states that $R\subseteq Q\x P$.
So $T$ is complete and has quantifier elimination. Let $L'=L\cup\set{c_{i}}{i<\omega}$
where $c_{i}$ are constants symbols, and let $T'$ be an expansion
of $T$ that says that $c_{i}\in P$ and $c_{i}\neq c_{j}$ for $i\neq j$.
So $T'$ is also complete and admits quantifier elimination. As $T$
has the independence property, so does $T'$.

Let $p\left(x\right)\in S^{1}\left(\emptyset\right)$ say that $x\in P$
and $x\neq c_{i}$ for all $i<\omega$. Finally, let $D$ be the finite
diagram $S^{<\omega}\left(\emptyset\right)\backslash\left\{ p\right\} $.
Easily $D$ is good (if $\C$ is a monster model of $T$, then let
$Q^{\C}\cup\set{c_{i}^{\C}}{i<\omega}$ be $\C_{D}$). It is easy
to see that $D$ is dependent.
\end{example}
Recall that a cardinal $\theta$ is \emph{strongly compact} if any
$\theta$-complete filter (with any domain) is contained in a $\theta$-complete
ultrafilter. For our context we will need to assume that if $D$ is
non-trivial, then there is a strongly compact cardinal $\theta>\left|T\right|$.
Strongly compact cardinals are measurable (see \cite[Corollary 4.2]{kanamori}).
Recall that a cardinal $\mu$ is \emph{measurable} if it is uncountable
and there is a $\mu$-complete non-principal ultrafilter on $\mu$.
It follows that there is a \emph{normal} such ultrafilter (i.e., closed
under diagonal intersection). See \cite[Exercise 5.12]{kanamori}.
Measurable cardinals are strongly inaccessible (see \cite[Theorem 2.8]{kanamori}),
which means that $\theta>\beth_{\left(2^{\left|T\right|}\right)^{+}}\geq\mu\left(D\right)$.
Fix some such $\theta$ throughout. If, however, $D$ is trivial,
then we do not need a strongly compact cardinal. 

We also note here a key fact about measurable cardinals that will
be useful later:
\begin{fact}
\label{fact:indiscernibles exist} \cite[Theorem 7.17]{kanamori}
Suppose that $\mu>\left|T\right|$ is a measurable cardinal and that
$\cal U$ is a normal (non-principal) ultrafilter on $\mu$. Suppose
that $\sequence{\bar{a}_{i}}{i<\mu}$ is a sequence of tuples in $\C$
of equal length $<\mu$, then for some set $X\in\cal U$, $\sequence{\bar{a}_{i}}{i\in X}$
is an indiscernible sequence. 
\end{fact}
As a consequence (which will also be used later), we have the following.
\begin{cor}
\label{cor:full-indiscernibles}If $A=\bigcup_{i<\mu}A_{i}\subseteq\C$
is a continuous increasing union of sets where $\left|A_{i}\right|<\mu$,
$B\subseteq\C$ is some set of cardinality $<\mu$, and $\sequence{\bar{a}_{i}}{i<\mu}$,
$\cal U$ are as in Fact \ref{fact:indiscernibles exist} with $\bar{a}_{i}$
tuples from $A$, then for some set $X\in\cal U$, $\sequence{\bar{a}_{i}}{i\in X}$
is \textup{fully indiscernible over $B$} (with respect to $A$ and
$\sequence{A_{i}}{i<\mu}$), which means that for every $i\in X$
and $j<i$ in $X$, we have $\bar{a}_{j}\subseteq A_{i}$ , and $\sequence{\bar{a}_{j}}{i\leq j\in X}$
is indiscernible over $A_{i}\cup B$. \end{cor}
\begin{proof}
This follows by the normality of the ultrafilter $\cal U$. First
note that if $E\subseteq\mu$ is a club then $E\in\cal U$ (why? Otherwise
$X=\mu\backslash E\in\cal U$, so the function $f:X\to\mu$ defined
by $\beta\mapsto\sup\left(\beta\cap E\right)$ is such that $f\left(\beta\right)<\beta$,
and by Fodor's lemma (which holds for normal ultrafilters), for some
$\gamma<\mu$ and $Y\subseteq X$ in $\cal U$, $f\upharpoonright Y=\gamma$
which easily leads to a contradiction). Hence the set $E=\set{i<\mu}{\forall j<i\left(\bar{a}_{j}\subseteq A_{i}\right)}$
is in $\cal U$. Furthermore, the set of limit ordinals $E'$ is also
in $\cal U$. The promised set $X$ is the intersection of $E\cap E'$
with the diagonal intersection of $X_{i}$ for $i<\mu$, where $X_{i}\in\cal U$
is such that $\sequence{\bar{a}_{i}}{i\in X_{i}}$ is indiscernible
over $A_{i}\cup B$ (which exists thanks to Fact \ref{fact:indiscernibles exist}).
Note that we have $\leq$ and not just $<$ when defining ``fully
indiscernible'', because $\sequence{A_{i}}{i<\mu}$ is continuous
and $X$ contains only limit ordinals. 
\end{proof}
The following demonstrates the need for Hanf numbers and strongly
compact cardinals.
\begin{lem}
\label{lem:eq formulations of dependence}For a finite diagram $D$
the following conditions are equivalent:
\begin{enumerate}
\item The formula $\phi\left(\bar{x},\bar{y}\right)$ has the independence
property.
\item For any $\lambda$ there is an indiscernible sequence $\sequence{\bar{a}_{i}}{i<\lambda}$
and $\bar{b}$ in $\C_{D}$ such that $\C_{D}\models\phi\left(\bar{a}_{i},\bar{b}\right)$
iff $i$ is even.
\item For any $\lambda$ there is a set $\set{\bar{a}_{i}}{i<\lambda}\subseteq\C_{D}$
such that for any $s\subseteq\lambda$ there is some $\bar{b}_{s}\in\C_{D}$
such that $\C_{D}\models\phi\left(\bar{a}_{i},\bar{b}_{s}\right)$
iff $i\in s$.
\item The same as (2) but with $\lambda=\theta$.
\item The same as (3) but with $\lambda=\theta$.
\end{enumerate}
\end{lem}
\begin{proof}
(1) $\Rightarrow$ (3): we may assume that $\lambda\geq\mu\left(D\right)$.
By assumption there is a sequence $\sequence{\bar{a}_{i}}{i<\mu\left(D\right)}$
and $\bar{b}$ in $\C_{D}$ as in the definition. Let $M\prec\C_{D}$
be a model of size $\mu\left(D\right)$ containing all these elements.
Add to the language $L$ new constants $\bar{c}$ in the length of
$\bar{b}$, a new predicate $P$ in the length of $\bar{x}$ and a
$2\lg\left(\bar{x}\right)$-ary symbol $<$, and a function symbol
$f$. Expand $M$ to $M'$, a structure of the expanded language,
by interpreting $\bar{c}^{M'}=\bar{b}$, $P^{M'}=\set{\bar{a}_{i}}{i<\mu\left(D\right)}$,
$\bar{a}_{i}<^{M'}\bar{a}_{j}$ iff $i<j$ and let $f^{M'}:P^{M'}\to M'$
be onto. 

Let $T_{0}=Th\left(M'\right)$. By assumption, $T_{0}$ has a $D$-model
of size $\mu\left(D\right)$, and so by definition $T_{0}$ has a
$D$-model $N'$ of cardinality $\lambda$ and we may assume that
its $L$-part $N$ is an elementary substructure of $\C_{D}$. So
the elements in $P^{N'}$, ordered by $<^{N'}$, form an $L$-indiscernible
sequence, and $\left|P^{N'}\right|=\lambda$. 

For convenience of notation, let $\left(I,<\right)$ be an order,
isomorphic to $\left(P^{N'},<^{N'}\right)$, and write $P^{N'}=\set{\bar{a}_{i}}{i\in I}$.
The order $<$ is discrete, so every $i\in I$ has a unique successor
$s\left(i\right)$, and $N\models\phi\left(\bar{c},\bar{a}_{i}\right)\leftrightarrow\neg\phi\left(\bar{c},\bar{a}_{s\left(i\right)}\right)$.
Let $Q=\set{i\in I}{N\models\phi\left(\bar{c},\bar{a}_{i}\right)}$,
so $\left|Q\right|=\lambda$. Then, by indiscernibility, for any
$R\subseteq Q$, 
\[
\sequence{\bar{a}_{i}}{i\in Q}\equiv\sequence{\bar{a}_{s^{R\left(i\right)}\left(i\right)}}{i\in Q}
\]
 where $R\left(i\right)=0$ iff $i\in R$, and $s^{0}=\id$, $s^{1}=s$.
Hence by the strong homogeneity of $\C_{D}$, $\set{\bar{a}_{i}}{i\in Q}$
satisfies (3). 

(2) $\Rightarrow$ (4), (3) $\Rightarrow$ (5), (4) $\Rightarrow$
(1): Obvious. 

(5) $\Rightarrow$ (2): We may assume that $\lambda\geq\theta$. Let
$\set{\bar{a}_{i}}{i<\theta}$ be as in (5). Since $\theta$ is measurable,
by Fact \ref{fact:indiscernibles exist}, we may assume that $\sequence{\bar{a}_{i}}{i<\theta}$
is indiscernible. By compactness we can extend this sequence to $\sequence{\bar{a}_{i}}{i<\lambda}$,
and let $A=\set{\bar{a}_{i}}{i<\lambda}$. Note that by indiscernibility,
the set containing all tuples in the new sequence is still a $D$-set,
so we may assume that this new sequence lies in $\C_{D}$. 

Let $O$ be the set of odd ordinals in $\lambda$. By indiscernibility
and homogeneity, for each $X\in\left[\lambda\right]^{<\theta}$ (i.e.,
$X\subseteq\lambda$, $\left|X\right|<\theta$) there is some $\bar{b}_{X}$
such that for all $i\in X$, $\C_{D}\models\phi\left(\bar{b}_{X},\bar{a}_{i}\right)$
iff $i\notin O$. By strong compactness, there is some $\theta$-complete
ultrafilter $\cal U$ on $\left[\lambda\right]^{<\theta}$ such that
for every $X\in I$ we have $\set{Y\in\left[\lambda\right]^{<\theta}}{X\subseteq Y}\in\cal U$.
Let $\bar{b}\models\average{\sequence{\bar{b}_{X}}{X\in\left[\lambda\right]^{<\theta}}}A{\cal U}$
which exists in $\C_{D}$ by Lemma \ref{lem:average type along ultrafilter},
then $\C_{D}\models\phi\left(\bar{b},\bar{a}_{i}\right)$ iff $i$
is even. 
\end{proof}
Dependence gives rise to the concept of the\emph{ }average type of
an indiscernible sequence, without resorting to ultrafilters. Let
$A\subseteq\C_{D}$, let $\alpha$ be an ordinal such that $\cof\left(\alpha\right)\geq\mu\left(D\right)$,
and let $\sequence{\bar{a}_{i}}{i<\alpha}$ be an indiscernible sequence
in $\C_{D}$. The \emph{average type} of $\sequence{\bar{a}_{i}}{i<\alpha}$
over $A$, denoted by $\average{\sequence{\bar{a}_{i}}{i<\alpha}}A{}$,
consists of formulas of the form $\phi\left(\bar{b},\bar{x}\right)$
with $\bar{b}\in A$, such that for some $i$, $\C_{D}\models\phi\left(\bar{b},\bar{a}_{j}\right)$
for every $j\ge i$. This is well defined as $\cof\left(\alpha\right)\geq\mu\left(D\right)$
(and as $D$ is dependent): otherwise, we can construct an increasing
unbounded sequence of ordinals $j_{i}<\alpha$, such that $\phi\left(\bar{b},\bar{a}_{j_{i}}\right)\leftrightarrow\neg\phi\left(\bar{b},\bar{a}_{j_{i+1}}\right)$,
and the length of this sequence is $\geq\mu\left(D\right)$. We show
that this type is indeed a $D$-type. 
\begin{lem}
\label{lem:average type on indiscernible sequence}Let $A\subseteq\C_{D}$
where $D$ is a dependent diagram, $\alpha$ an ordinal such that
$\cof\left(\alpha\right)\geq\mu\left(D\right)+\left|T\right|^{+}$,
and let $\sequence{\bar{a}_{i}}{i<\alpha}$ be an indiscernible sequence
in $\C_{D}$. The average type $r=\average{\sequence{\bar{a}_{i}}{i<\alpha}}A{}$
is a $D$-type. \end{lem}
\begin{proof}
The proof is similar to that of Lemma \ref{lem:average type along ultrafilter},
but here we use the fact that the end-segment filter on $\alpha$
is $\cof\left(\alpha\right)$-complete. The main point is that for
a formula $\varphi\left(\bar{x},\bar{a}\right)\in r$, there is some
$j<\alpha$ such that $\varphi\left(\bar{a}_{i},\bar{a}\right)$ holds
for all $i>j$. 
\end{proof}

\section{\label{sec:The-generic-pair}The generic pair conjecture}

From this section onwards, fix a dependent diagram $D$. We also
fix a strongly compact cardinal $\theta>\left|T\right|$. When $D$
is trivial, there is no need for strong compact cardinals, and one
can assume $\theta=\left|T\right|^{+}$, and replace $<\theta$ satisfiable
by finitely satisfiable. We leave it to the reader to find  the precise
replacement. 
\begin{conjecture}
\label{conj:The generic pair conjecture}(The generic pair conjecture)
Suppose $D$ is dependent. Assume $\theta<\lambda=\lambda^{<\lambda}$
and $\lambda^{+}=2^{\lambda}$. Let $\bar{M}=\langle M_{\alpha}:\alpha<\lambda^{+}\rangle$
be an increasing continuous sequence of elementary substructures of
$\C_{D}$ of cardinality $\lambda$, such that ${\bf M}=\bigcup_{\alpha<\lambda^{+}}M_{\alpha}$
is $D$-saturated of size $\lambda^{+}$. 

Then there exists a club $E\subseteq\lambda^{+}$ such that
\begin{itemize}
\item if $\alpha_{1}<\beta_{1},\alpha_{2}<\beta_{2}\in E$ are all of cofinality
$\lambda$, then $\left(M_{\beta_{1}},M_{\alpha_{1}}\right)\cong\left(M_{\beta_{2}},M_{\alpha_{2}}\right)$.
\end{itemize}
\end{conjecture}
To give some motivation, note that it is easy to find a club $\Esat\subseteq\lambda^{+}$
such that for any $\delta\in\Esat$ of cofinality $\lambda$, $M_{\delta}$
is homogenous and $D\left(M_{\delta}\right)=D$ (equivalently $D$-saturated
by Lemma \ref{lem:Grossberg}). Just let $\Esat$ be the set of ordinals
$\delta<\lambda^{+}$ such that for any $\alpha<\delta$, every $p\in S_{D}^{1}\left(A\right)$
for any $A\subseteq M_{\alpha}$ of size $<\lambda$ is realized in
$M_{\delta}$. Then for any $\delta\in\Esat$ of cofinality $\lambda$,
$M_{\delta}$ is $D$-saturated, and any such two are isomorphic (see
Corollary \ref{cor:homo-isom}). 

In this section we will outline the proof of Conjecture \ref{conj:The generic pair conjecture}
under the assumption that a ``good family of decompositions'' exists.

We call a tuple of the form $\xx=\left(M_{\xx},B_{\xx},\bar{d}_{\xx},\bar{c}_{\xx},r_{\xx}\right)$
a \emph{$\lambda$-decomposition}%
\footnote{The idea behind the name ``decomposition'' will be clearer later,
where this notion is used to analyze the type of $\bar{d}$ over $M$. %
} when $\left|M_{\xx}\right|=\lambda$ and $M_{\xx}\subseteq\C$ is
a $D$-model, $B_{\xx}\subseteq M_{\xx}$ has cardinality $<\lambda$,
$\bar{c}_{\xx},\bar{d}_{\xx}\in\mathfrak{C}_{D}^{<\lambda}$ and $r_{\xx}\in S_{D}^{<\lambda}\left(\emptyset\right)$
is a complete type in variables $\left(\bar{x}_{\bar{c}_{\xx}},\bar{x}_{\bar{d}_{\xx}},\bar{x}_{\bar{c}_{\xx}}',\bar{x}_{\bar{d}_{\xx}}'\right)$
(where $\bar{x}_{\bar{d}_{\xx}},\bar{x}_{\bar{d}_{\xx}}'$ have the
same length as $\bar{d}_{\xx}$, etc.). 

An \emph{isomorphism} between two $\lambda$-decompositions $\xx$
and $\yy$ is just an elementary map with domain $M_{\xx}\cup\bigcup\bar{c}_{\xx}\cup\bigcup\bar{d}_{\xx}$
which maps all the ingredients of $\xx$ onto those of $\yy$, and
in particular, if $\xx\cong\yy$ then $r_{\xx}=r_{\yy}$. A \emph{weak
isomorphism} between $\xx$ and $\yy$ is a restriction of an isomorphism
to $\left(B_{\xx},\bar{d}_{\xx},\bar{c}_{\xx},r_{\xx}\right)$ (so
there exists some isomorphism extending it). We write $\xx\leq\yy$
when $M_{\xx}=M_{\yy}$, $B_{\xx}\subseteq B_{\yy}$, $r_{\xx}\subseteq r_{\yy}$
(i.e., $r_{\yy}$ may add more information on the added variables),
$\bar{c}_{\xx}\trianglelefteq\bar{c}_{\yy}$ (i.e., $\bar{c}_{\xx}$
is an initial segment of $\bar{c}_{\yy}$) and $\bar{d}_{\xx}\trianglelefteq\bar{d}_{\yy}$.
If $\xx$ and $\yy$ are $\lambda$-decompositions with $M_{\xx}=M_{\yy}$
such that for some $\zz$, $\zz\leq\xx,\yy$, we will say that they
are isomorphic over $\zz$ if there is an isomorphism from $\xx$
to $\yy$ fixing $\bar{d}_{\zz},\bar{c}_{\zz},B_{\zz}$. 
\begin{defn}
\label{def:good family}(A good family) A family $\Ff$ of $\lambda$-decompositions
is \emph{good} when: 
\begin{enumerate}
\item \label{enu:invariant under isom}The family $\Ff$ is invariant under
isomorphisms. 
\item \label{enu:every model is saturated}For every $\xx\in\Ff$, $M_{\xx}$
is $D$-saturated. 
\item \label{enu:non-empty}For every $D$-saturated $M\prec\C_{D}$ of
size $\lambda$, the ``trivial decomposition'' $\left(M,\emptyset,\emptyset,\emptyset,\emptyset\right)\in\mathfrak{F}$. 
\item \label{enu:enlarging}For every $\xx\in\mathfrak{F}$ and $\bar{d}\in\mathfrak{C}_{D}^{<\lambda}$
there exists some $\yy\in\mathfrak{F}$ such that $\xx\le\yy$, and
$\bar{d}_{\yy}\trianglerighteq\bar{d}_{\xx}\bar{d}$.
\item \label{enu:enbase}For every $\xx\in\mathfrak{F}$ and $b\in M_{\xx}$,
$\left(M_{\xx},B_{\xx}\cup\left\{ b\right\} ,\bar{d}_{\xx},\bar{c}_{\xx},r_{\xx}\right)\in\Ff$.
\item \label{enu:iso-extension}Suppose that $\xx_{1},\xx_{2},\yy_{1}\in\mathfrak{F}$
where $\xx_{1}\le\yy_{1}$ and there exists some isomorphism $f:\xx_{1}\to\xx_{2}$,
then there exists some $\yy_{2}\in\mathfrak{F}$ such that $\xx_{2}\le\yy_{2}$
and $f$ can be extended to an isomorphism $\yy_{1}\to\yy_{2}$. 
\item \label{enu:union}Suppose that $\sequence{\xx_{i}}{i<\delta}$ is
a sequence of $\lambda$-decompositions from $\mathfrak{F}$ such
that $\delta<\lambda$ is a limit ordinal and for every $i<j<\delta$
we have $\xx_{i}\le\xx_{j}$, then $\xx_{\delta}=\sup_{i<\delta}\xx_{i}=\left(M,\bigcup_{i<\delta}B_{\xx_{i}},\bigcup_{i<\delta}\bar{d}_{\xx_{i}},\bigcup_{i<\delta}\bar{c}_{\xx_{i}},\bigcup_{i<\delta}r_{\xx_{i}}\right)\in\mathfrak{F}$
. Note that as $\lambda$ is regular and $\delta<\lambda$ this makes
sense. 
\item \label{enu:isohomo} Suppose that $\sequence{\xx_{i}}{i<\delta}$
and $\sequence{\yy_{i}}{i<\delta}$ are increasing sequences of $\lambda$-decompositions
from $\mathfrak{F}$ such that $\delta<\lambda$ is a limit ordinal
and for each $i<\delta$ there is a weak isomorphism $g_{i}:\xx_{i}\to\yy_{i}$
such that $g_{i}\subseteq g_{j}$ whenever $i<j$. Then the union
$\bigcup_{i<\delta}g_{i}$ is a weak isomorphism from $\xx=\sup_{i<\delta}\xx_{i}$
to $\yy=\sup_{i<\delta}\yy_{i}$. 
\item \label{enu:count} For every $D$-model $M$ of cardinality $\lambda$,
the number of $\xx\in\Ff$ with $M_{\xx}=M$ up to isomorphism is
$\leq\lambda$. 
\end{enumerate}
\end{defn}
\begin{rem}
The roles of $\bar{c}_{\xx}$ and $r_{\xx}$ will become crucial in
the next sections. In this section it is important in order to restrict
the class of isomorphisms. 
\end{rem}

\begin{rem}
In Definition \ref{def:good family}, (\ref{enu:iso-extension}) follows
from (\ref{enu:invariant under isom}). 
\end{rem}

\begin{rem}
\label{rem:everything in M}Note that by point (\ref{enu:invariant under isom})
in Definition \ref{def:good family}, and as ${\bf M}$ is $D$-saturated
of cardinality $\lambda^{+}$, if $\Ff$ is good, then $\Ff$ is also
good when we restrict it to decompositions contained in ${\bf M}$
(rather than $\C_{D}$). More precisely, in points (\ref{enu:enlarging})
and (\ref{enu:iso-extension}), the promised decompositions $\yy$
and $\yy_{2}$ respectively can be found in ${\bf M}$ if the given
decompositions ($\xx,$ $\xx_{1}$, $\xx_{2}$ and $\yy_{1}$) are
in ${\bf M}$.
\end{rem}
Let us give an example of a ``baby application'' of the existence
of a good family before we delve into the generic pair conjecture.
This next theorem is a weak version of \cite[Conclusion 3.13]{Sh950}. 
\begin{thm}
\label{thm:clever counting of types}Suppose $\Ff$ is a good family.
Then, for a $D$-saturated model $M$ of size $\lambda$, the number
of types in $S_{D}^{<\lambda}\left(M\right)$ up to conjugation is
$\leq\lambda$. \end{thm}
\begin{proof}
Suppose $\gamma<\lambda$, and $\sequence{p_{i}}{i<\lambda^{+}}$
is a sequence of types in $S_{D}^{\gamma}\left(M\right)$, which are
pairwise non-conjugate. Let $\bar{d}_{i}\models p_{i}$. By (\ref{enu:enlarging})
in Definition \ref{def:good family}, for some $\xx_{i}\in\Ff$, $\bar{d}_{i}\init\bar{d}_{\xx_{i}}$.
Obviously, for $i\neq j$, $\tp\left(\bar{d}_{\xx_{i}}/M\right)$
and $\tp\left(\bar{d}_{\xx_{j}}/M\right)$ are not conjugates. But
according to (\ref{enu:count}), this is impossible. \end{proof}
\begin{rem}
\label{rem:Upto Isomorphism over z}Suppose $\zz$ is a $\lambda$-decomposition.
From (\ref{enu:count}) in Definition \ref{def:good family} it follows
that the number of $\xx\in\Ff$ such that $\zz\leq\xx$ up to isomorphism
over $\zz$ is $\leq\lambda$. Indeed, if not there is a sequence
$\sequence{\xx_{i}}{i<\lambda^{+}}$ of $\lambda$-decompositions
in $\Ff$ containing $\zz$ which are pairwise not isomorphic over
$\zz$. By (\ref{enu:count}), we may assume that they are pairwise
isomorphic, and let $f_{i}:\xx_{i}\to\xx_{0}$ be isomorphisms. So
$f_{i}$ must fix $\bar{d}_{\zz}$ and $\bar{c}_{\zz}$ as they are
initial segments. In addition, $f_{i}\upharpoonright B_{\zz}$ is
a sequence of length $<\lambda$ of elements in $M_{\zz}$, and there
are $\lambda$ such sequences (as $\lambda^{<\lambda}=\lambda$),
so for some $i\neq j$, $f_{i}\upharpoonright B_{\zz}=f_{j}\upharpoonright B_{\zz}$.
Hence $f_{i}^{-1}\circ f_{j}\upharpoonright B_{\zz}=\id$ --- contradiction.
\end{rem}
For a decomposition $\xx$, we will write $\xx\Subset M$ for $M_{\xx}\subseteq M$
and $\left(\bar{c}_{\xx},\bar{d}_{\xx}\right)\in\left(M^{<\lambda}\right)^{2}$. 
\begin{defn}
\label{def:complete+bnf}Let $\gamma<\lambda^{+}$, and let $\mathfrak{F}$
be a good family of $\lambda$-decompositions.
\begin{enumerate}
\item We say that $\gamma$ is \emph{$\mathfrak{F}$-complete} if for every
$\alpha<\beta<\gamma$ such that $M_{\alpha}$ is $D$-saturated,
$\yy\in\Ff$ with $M_{\yy}=M_{\alpha}$ and $\bar{d}\in M_{\beta}^{<\lambda}$
such that $\yy\Subset M_{\beta}$, there exists some $\yy\leq\xx\in\Ff$
such that $\bar{d}_{\xx}\trianglerighteq\bar{d}\bar{d}_{\yy}$ and
$\xx\Subset M_{\gamma}$.
\item We say that $\gamma$ is \emph{$\Ff$-representative} if for every
$\alpha<\beta<\gamma$ such that $M_{\alpha}$ is $D$-saturated,
$\yy\in\Ff$ with $M_{\yy}=M_{\alpha}$ and every $\lambda$-decomposition
$\zz$ over $M_{\alpha}$ such that $\zz\Subset M_{\beta}$ and $\zz\leq\yy$,
there exists $\xx\in\Ff$ such that $M_{\xx}=M_{\alpha}$, $\xx\Subset M_{\gamma}$,
$\zz\leq\xx$ and $\xx$ is isomorphic to $\yy$ over $\zz$. 
\end{enumerate}
\end{defn}
\begin{prop}
\label{prop:F-complete club} Let $\mathfrak{F}$ be a family of good
$\lambda$-decompositions. Let $\Ecom\subseteq\lambda^{+}$ be the
set of all $\delta<\lambda^{+}$ which are $\mathfrak{F}$-complete.
Then $\Ecom$ is a club. \end{prop}
\begin{proof}
The fact that $\Ecom$ is a closed is easy. Suppose $\beta<\lambda^{+}$.
Let $\beta<\beta'<\lambda^{+}$ be such that for every $\alpha<\beta$
such that $M_{\alpha}$ is $D$-saturated, and every $\bar{d}\in M_{\beta}^{<\lambda}$
and $\yy\in\Ff$ with $\yy\Subset M_{\beta}$ and $M_{\yy}=M_{\alpha}$,
there is some $\yy\leq\xx\in\Ff$ such that $\bar{d}_{\xx}\fini\bar{d}_{\yy}\bar{d}$,
$M_{\xx}=M_{\alpha}$ and $\xx\Subset M_{\beta'}$. The ordinal $\beta'$
exists because $\lambda^{<\lambda}=\lambda$ (so the number of $\yy$'s
and the number of $\bar{d}$'s is $\leq\lambda$), by (\ref{enu:enlarging})
of Definition \ref{def:good family} and by Remark \ref{rem:everything in M}.
By induction, we can thus define an increasing sequence of ordinals
$\beta_{i}$ for $i<\omega$ where $\beta_{0}=\beta$ and $\beta_{i+1}=\beta_{i}'$.
Finally, $\gamma=\beta_{\omega}\in\Ecom$. \end{proof}
\begin{prop}
\label{prop:BnF club}Let $\Ff$ be a family of good $\lambda$-decompositions.
Let $\Ebnf\subseteq\lambda^{+}$ be the set of all $\delta<\lambda^{+}$which
are $\Ff$-representative. Then $\Ebnf$ is a club.\end{prop}
\begin{proof}
The proof is similar to the proof of Proposition \ref{prop:F-complete club},
but now in order to show that $\Ebnf$ is unbounded, we use Remark
\ref{rem:Upto Isomorphism over z}.\end{proof}
\begin{thm}
\label{thm:generic pair with good family}Suppose $\Ff$ is a good
family. Let $E=\Esat\cap\Ebnf\cap\Ecom\subseteq\lambda^{+}$. This
is a club. For every $\alpha_{1}<\beta{}_{1},\alpha_{2}<\beta_{2}\in E$
of cofinality $\lambda$ we have $\left(M_{\beta_{1}},M_{\alpha_{1}}\right)\cong\left(M_{\beta_{2}},M_{\alpha_{2}}\right)$
. Hence Conjecture \ref{conj:The generic pair conjecture} holds. \end{thm}
\begin{proof}
Let $AP$%
\footnote{AP stands for approximations.%
} be the collection of tuples of the form $p=\left(\xx_{p},\yy_{p},h_{p}\right)=\left(\xx,\yy,h\right)$
where $\xx,\yy\in\Ff$ and $h:\xx\to\yy$ is a weak isomorphism, such
that $M_{\xx}=M_{\alpha_{1}}$, $\xx\Subset M_{\beta_{1}}$, $M_{\yy}=M_{\alpha_{2}}$
and $\yy\Subset M_{\beta_{2}}$. For every $p_{1},p_{2}\in AP$ we
write $p_{1}\le_{AP}p_{2}$ if $\xx_{p_{1}}\leq\xx_{p_{2}}$, $\yy_{p_{1}}\leq\yy_{p_{2}}$
, and $h_{p_{1}}\subseteq h_{p_{2}}$. 

We proceed to construct an isomorphism by a back and forth argument.
In the forth part, we may add an element from $M_{\alpha_{1}}$ to
$B_{\xx}$ (thus increasing the $M_{\alpha_{1}}$-part of the domain
of $h$), or an element from $M_{\beta_{1}}$ to $\bar{d}_{\xx}$
(thus increasing the $M_{\beta_{1}}$-part). We also have to take
care of the limit stage. 

As one could take $p$ to be a trivial tuple by (\ref{enu:non-empty})
in Definition \ref{def:good family}, and as $\alpha_{1},\alpha_{1}\in\Esat$
(and their cofinality is $\lambda$ so that $M_{\alpha_{1}},M_{\alpha_{2}}$
are saturated), $AP\neq\emptyset$. 

Adding an element from $M_{\alpha_{1}}$: let $p\in AP$ and $a\in M_{\alpha_{1}}$.
As $h$ is a weak isomorphism, there is some isomorphism $h^{+}:\xx\to\yy$
extending $h$. Let $h^{+}\left(a\right)=b\in M_{\alpha_{2}}$. Thus,
by (\ref{enu:enbase}) in Definition \ref{def:good family}, we may
define $p'=\left(\xx',\yy',h'\right)$ by adding $a$ to $B_{\xx}$
and $b$ to $B_{\yy}$, and defining $h'=h\cup\left\{ \left(a,b\right)\right\} $.
Of course, $h'$ is still a weak isomorphism as witnessed by the same
$h^{+}$. It follows that $p\le_{AP}p'$. 

Adding an element from $M_{\beta_{1}}$: let $d\in M_{\beta_{1}}$
and $p\in AP$. Since $\mathfrak{F}$ is good, $\alpha_{1}\in\Esat$,
$\beta_{1}\in\Ecom$, and by (\ref{enu:enlarging}) in Definition
\ref{def:good family}, there is some $\xx\leq\xx'\in\Ff$ such that
$\bar{d}_{\xx}d\trianglelefteq\bar{d}_{\xx'}$ and $\xx'\Subset M_{\beta_{1}}$
(here we also used the fact that the cofinality of $\beta_{1}$ is
$\lambda$). 

Let $h^{+}:\xx\to\yy$ be as above. By (\ref{enu:iso-extension})
in Definition \ref{def:good family}, $h^{+}$ extends to an isomorphism
$h^{++}:\xx'\to\yy'$ for some $\yy'\in\mathfrak{F}$, such that $\yy\leq\yy'$
(and we may also assume that $\yy$ is contained in ${\bf M}$ by
Remark \ref{rem:everything in M}).

Since $\beta_{2}\in\Ebnf$ (and since its cofinality is $\lambda$),
there exists some $\yy''\in\mathfrak{F}$ such that $\yy''\Subset M_{\beta_{2}}$,
$\yy\leq\yy''$, and $\yy''$ is isomorphic to $\yy'$ over $\yy$,
as witnessed by $f:\yy'\to\yy''$ (in particular $f\upharpoonright M_{\alpha_{2}}$
is an automorphism of $M_{\alpha_{2}}$). We have then $p'=\left(\xx',\yy'',\left(f\circ h^{++}\right)\upharpoonright\left(B_{\xx'},\bar{d}_{\xx'},\bar{c}_{\xx'},r_{\xx'}\right)\right)\in\mathfrak{F}$
satisfies that $p\le_{AP}p'$ and $d\in\bar{d}_{\xx'}$. 

Of course we must also switch the roles of $\xx$ and $\yy$ in the
above steps. 

The limit stage: suppose $\sequence{p_{i}}{i<\delta}$ is an increasing
sequence of approximation where $\delta<\lambda$ is some limit. Let
\[
p=\sup_{i<\delta}p_{i}=\left(\sup_{i<\delta}\xx_{p_{i}},\sup_{i<\delta}\yy_{p_{i}},\bigcup_{i<\delta}h_{i}\right).
\]
This tuple is still in $AP$ by (\ref{enu:union}) and (\ref{enu:isohomo})
in Definition \ref{def:good family}. 
\end{proof}

\section{\label{sec:The-type-decomposition}type decompositions}

Section \ref{sec:The-generic-pair} gave the proof of the generic
pair conjecture (Conjecture \ref{conj:The generic pair conjecture})
by using $\lambda$-decompositions and a good family of these (Definition
\ref{def:good family}). Here we will start to construct what eventually
will be the good family. For this we need to define two kinds of decompositions.
The first is the \emph{tree-type decomposition} (explained in Subsection
\ref{sub:Tree-type-decomposition}), which is the basic building block
of the \emph{self-solvable decomposition} which will be introduced
in Subsection \ref{sub:Self-solvable-decomposition}. Eventually,
the good family will be the family of self-solvable decompositions. 

As usual, we assume that $\theta>\left|T\right|$ is a strongly compact
cardinal (unless $D$ is trivial and then $\theta=\left|T\right|^{+}$,
and also replace $<\theta$ satisfiable by finitely satisfiable when
appropriate, see the beginning of Section \ref{sec:The-generic-pair}),
and that $D$ is dependent. Also, assume that $\lambda=\lambda^{<\lambda}>\theta$.

\subsection{\label{sub:Tree-type-decomposition}Tree-type decomposition}
\begin{defn}
\label{def:tree type}Let $M\prec\mathfrak{C}_{D}$ be a $D$-model
of cardinality $\lambda$. A \emph{$\lambda$-tree-type decomposition}
is a $\lambda$-decomposition $\left(M,B,\bar{d},\bar{c},r\right)$
with the following properties:
\begin{enumerate}
\item The tuple $\bar{c}$ is of length $<\kappa=\theta+\left|\lg\left(\bar{d}\right)\right|^{+}$
and the type $\tp\left(\bar{c}/M\right)$ does not split over $B$.
See also Remark \ref{rem:f.s. instead of invariant}.
\item For every $A\subseteq M$ such that $\left|A\right|<\lambda$ there
exists some $\bar{e}_{A}\in M^{<\kappa}$ such that $\tp\left(\bar{d}/\bar{e}_{A}+\bar{c}\right)\vdash\tp\left(\bar{d}/A+\bar{c}\right)$.
By this we mean that if $\bar{d}'\in\C_{D}^{<\lambda}$ realizes the
same type as $\bar{d}$ over $\bar{e}_{A}+\bar{c}$ (which we denote
by $\bar{d}'\equiv_{\bar{e}_{A}\bar{c}}\bar{d}$), then $\bar{d}'\equiv_{A\bar{c}}\bar{d}$.
Note: we do not ask that this is true in $\C$, only in $\C_{D}$.
\end{enumerate}
\end{defn}
\begin{rem}
Why ``tree-type''? if $\xx$ is a tree-type decomposition such that
for simplicity $\lg\left(\bar{d}\right)<\theta$, then we may define
a partial order on $M^{<\theta}$ by $\bar{e}_{1}\leq\bar{e}_{2}$
if $\tp\left(\bar{d}/\bar{c}+\bar{e}_{2}\right)\vdash\tp\left(\bar{d}/\bar{c}+\bar{e}_{1}\right)$.
Then this order is $\lambda$-directed (so looks like a tree in some
sense). 
\end{rem}

\begin{rem}
If $\tp\left(\bar{d}/M\right)$ does not split over a $B$ (where
$\left|B\right|<\lambda$ as usual), then $\left(M,B,\bar{d},\bar{d},r\right)$
is a $\lambda$-tree-type decomposition for any $r$: in (2) take
$\bar{e}_{A}=\emptyset$. 
\end{rem}

\begin{rem}
\label{rem:f.s. instead of invariant}In Definition \ref{def:tree type}
(1), we could ask that $\tp\left(\bar{c}/M\right)$ is $<\theta$
satisfiable in $B$ in the sense that any $<\theta$ formulas from
this type \uline{in finitely many variables} are realized in $B$. 
\end{rem}

\begin{rem}
In this section, the role of $\bar{c}$ becomes clearer, but $r$
will not have any role. \end{rem}
\begin{example}
\cite[Exercise 2.18]{Sh:900} In DLO --- the theory of $\left(\Qq,<\right)$
--- suppose $M$ is a saturated model of cardinality $\lambda$, and
$d\in\C\backslash M$ is some point. Let $C_{1}$, $C_{2}$ be the
corresponding left and right cuts that $d$ determines in $M$. As
$M$ is saturated at least one of these cuts has cofinality $\lambda$.
If only one has, then $\tp\left(d/M\right)$ does not split over the
smaller cut, so $\left(M,C_{i},d,d,\emptyset\right)$ is a tree-type
decomposition for $i=1$ or $i=2$. Otherwise for each $A$ of cardinality
$<\lambda$, there are $e_{1}<d<e_{2}$ in $M$ such that $C_{1}\cap A<e_{1}<e_{2}<C_{2}\cap A$,
so $\tp\left(d/e_{1}e_{2}\right)\vdash\tp\left(d/e_{1}e_{2}A\right)$.
In this case, $\left(M,\emptyset,d,\emptyset,\emptyset\right)$ is
a tree-type decomposition. 
\end{example}
Our aim now is to prove that when $M$ is a $D$-model, then for every
$\bar{d}\in\mathfrak{C}_{D}^{<\lambda}$ there exists a tree-type
decomposition $\xx$ such that $\bar{d}=\bar{d}_{\xx}$. In fact,
we can start with any tree-type decomposition $\xx$, for instance
the trivial one $\left(M,\emptyset,\emptyset,\emptyset,\emptyset\right)$,
and find some tree-type decomposition $\yy\geq\xx$ such that $\bar{d}_{\yy}=\bar{d}_{\xx}\bar{d}$.
In a sense, we decompose the type of $\bar{d}$ over $M$ into two
parts: the invariant one and the ``tree-like'' one. 
\begin{defn}
\label{def:K}Let $M\prec\C_{D}$ be of size $\lambda$, $\bar{d}\in\C_{D}^{<\lambda}$
and $C\subseteq\C_{D}$ be of size $<\kappa=\left|\lg\left(\bar{d}\right)\right|^{+}+\theta<\lambda$.
The class $\K$ contains all pairs $\aa=\left(B_{\aa},\bar{c}_{\aa}\right)=\left(B,\bar{c}\right)$
such that: 
\begin{enumerate}
\item $\bar{c}=\sequence{\left(\bar{c}_{i,0},\bar{c}_{i,1}\right)}{i<\gamma}\in\left(\mathfrak{C}_{D}^{<\omega}\times\mathfrak{C}_{D}^{<\omega}\right)^{\gamma}$,
and $B\subseteq M$, $|B|<\lambda$. 
\item $\gamma<\kappa$. 
\item For all $i<\gamma$, $\tp\left(\bar{c}_{i}/MC+\bar{c}_{<i}\right)$
is $<\theta$ satisfiable in $B$ where $\bar{c}_{i}$ is $\bar{c}_{i,0}\concat\bar{c}_{i,1}$.
Abusing notation, we identify $\bar{c}$ with the concatenation of
$\bar{c}_{i}$ for $i<\gamma$. It follows that $\tp\left(\bar{c}/MC\right)$
does not split over $B$. 
\item For every $i<\gamma$, $\tp\left(\bar{c}_{i,0}/MC+\bar{c}_{<i}\right)=\tp\left(\bar{c}_{i,1}/MC+\bar{c}_{<i}\right)$
and in particular they are of the same (finite) length, and $\tp\left(\bar{c}_{i,0}/MC+\bar{c}_{<i}+\bar{d}\right)\neq\tp\left(\bar{c}_{i,1}/MC+\bar{c}_{<i}+\bar{d}\right)$.
\end{enumerate}
The class $\mxK$ consists of all the maximal elements in $\K$ with
respect to the order $<$ defined by $\aa<\bb$ iff $B_{\aa}\subseteq B_{\bb}$,
$\bar{c}_{\aa}\init\bar{c}_{\bb}$ and $\bar{c}_{\aa}\neq\bar{c}_{\bb}$.
That is, it contains all $\aa\in\K$ such that there is no $\bb\in\K$
with $B_{\aa}\subseteq B_{\bb}$ and $\bar{c}_{\aa}$ is a strict
first segment of $\bar{c}_{\bb}$.\end{defn}
\begin{thm}
\label{thm:Maximum exists}For every $\bar{d}\in\mathfrak{C}_{D}^{<\lambda}$,
$C$ and $M$ as in Definition \ref{def:K}, if $\aa\in\K$ then there
exists some $\bb\in\mxK$ such that $\aa\leq\bb$. \end{thm}
\begin{proof}
Let $\bar{c}=\bar{c}_{\aa}=\sequence{\left(\bar{c}_{i,0},\bar{c}_{i,1}\right)}{i<\gamma}$.
We try to construct an increasing sequence $\sequence{\aa_{\alpha}}{\gamma\leq\alpha<\kappa}$
of elements in $\K$, where $\kappa=\left|\lg\left(\bar{d}\right)\right|^{+}+\theta<\lambda$,
as follows:
\begin{enumerate}
\item $\aa_{\gamma}=\aa$.
\item If $\alpha$ is limit then $\aa_{\alpha}=\sup_{\beta<\alpha}\aa_{\beta}$,
i.e., $B_{\aa_{\alpha}}=\bigcup_{\beta<\alpha}B_{\aa_{\beta}}$ and
$\bar{c}_{\aa_{\alpha}}=\bigcup_{\beta<\alpha}\bar{c}_{\aa_{\beta}}$.
Note that this is well defined, i.e., $\aa_{\alpha}\in\K$. 
\item Suppose $\alpha=\beta+1$ and $\aa_{\beta}$ has been constructed.
Let 
\[
\aa_{\alpha}=\left(B_{\alpha},\bar{c}_{\aa_{\beta}}\concat\left(\bar{c}_{\beta,0},\bar{c}_{\beta,1}\right)\right)
\]
just in case there are $\bar{c}_{\beta,0},\bar{c}_{\beta,1}\in\mathfrak{C}_{D}^{<\omega}$,
$B_{\aa_{\beta}}\subseteq B_{\alpha}\subseteq M$ such that $\aa_{\alpha}\in\K$. 
\end{enumerate}
If we got stuck somewhere in the construction it must be in the successor
stage $\alpha$, and then $\aa_{\alpha}\in\mxK$ is as requested.
So suppose we succeed: we constructed $\sequence{\left(\bar{c}_{\alpha,0},\bar{c}_{\alpha,1}\right)}{\alpha<\kappa}$.
As usual we denote $\bar{c}_{\alpha}=\bar{c}_{\alpha,0}\concat\bar{c}_{\alpha,1}$. 

By the definition of $\K$, it follows that for every $\alpha<\kappa$,
there are $\bar{a}_{\alpha}\in A^{<\omega}$ where $A=MC$, $\bar{b}_{\alpha}\in C_{\alpha}^{<\omega}$
where $C_{\alpha}=\bigcup_{\beta<\alpha}\bar{c}_{\beta}$, and a formula
$\varphi_{\alpha}\left(\bar{x}_{\bar{d}},\bar{w}_{\alpha},\bar{y}_{\alpha},\bar{z}_{\alpha}\right)$
such that $\C_{D}\models\varphi_{\alpha}\left(\bar{d},\bar{c}_{\alpha,0},\bar{a}_{\alpha},\bar{b}_{\alpha}\right)$
but $\C_{D}\models\neg\varphi_{\alpha}\left(\bar{d},\bar{c}_{\alpha,1},\bar{a}_{\alpha},\bar{b}_{\alpha}\right)$.
(The variables are all in the appropriate length, but only finitely
many of them appear in the formula.)

For every $\alpha<\kappa$, let $f\left(\alpha\right)$ be the maximal
ordinal $<\alpha$ such that $\bar{b}_{\alpha}$ intersects $\bar{c}_{f\left(\alpha\right)}$.
By Fodor's Lemma, There exists some cofinal set $S\subseteq\kappa$
and $\beta<\kappa$ such that for every $\alpha\in S$ we have $f\left(\alpha\right)=\beta$.
By restricting to a smaller set, we may assume that for any $\alpha\in S$,
$\alpha>\beta$ and $\varphi_{\alpha}=\varphi$ is constant. 

As $\bar{c}_{\alpha,0}\equiv_{A\bar{c}_{<\alpha}}\bar{c}_{\alpha,1}$
and as $\tp\left(\bar{c}_{\alpha}/A+\bar{c}_{<\alpha}\right)$ does
not split over $A$, it follows that $\tp\left(\sequence{\bar{c}_{\alpha,\eta\left(\alpha\right)}}{\alpha\in S}/AC_{\beta+1}\right)$
does not depend on $\eta$ when $\eta:S\to2$. To prove this it is
enough to consider a finite subset $S_{0}\subseteq S$, and to prove
it by induction on its size. Indeed, given $S_{0}=\left\{ \alpha_{0}<\ldots<\alpha_{n+1}\right\} $,
and any $\eta:S_{0}\to2$, 
\begin{align*}
\sequence{\bar{c}_{\alpha,\eta\left(\alpha\right)}}{\alpha\in S_{0}}\equiv_{AC_{\beta+1}} & \sequence{\bar{c}_{\alpha,\eta\left(\alpha\right)}}{\alpha\in S_{0}\backslash\left\{ \alpha_{n+1}\right\} }\concat\left\langle \bar{c}_{\alpha_{n+1},0}\right\rangle \\
\equiv_{AC_{\beta+1}} & \sequence{\bar{c}_{\alpha,0}}{\alpha\in S_{0}}.
\end{align*}

It follows by homogeneity that for any subset $R$ of $S$ there is
some $\bar{d}_{R}\in\C_{D}^{<\lambda}$ such that $\C_{D}\models\varphi\left(\bar{d}_{R},\bar{c}_{\alpha,0},\bar{a}_{\alpha},\bar{b}_{\alpha}\right)$
iff $\alpha\in R$. But this is a contradiction to the fact that $D$
is dependent, see Lemma \ref{lem:eq formulations of dependence} (5).\end{proof}
\begin{defn}
\label{def:orthogonal}Suppose $p\left(\bar{x}\right),q\left(\bar{y}\right)\in S_{D}\left(A\right)$
for some $A\subseteq\C_{D}$. We say that $p$ is \emph{orthogonal}%
\footnote{Usually this notion is called \emph{weakly orthogonal}, as the notion
of orthogonal types already has meaning in stable theories. However
here we have no room for confusion, so we decided to stick with the
simpler term.%
} to $q$ if there is a unique $r\left(\bar{x},\bar{y}\right)\in S_{D}\left(A\right)$
which extends $p\left(\bar{x}\right)\cup q\left(\bar{y}\right)$. 
\end{defn}

\begin{defn}
\label{def:tree like}Suppose that $M\prec\C_{D}$, and $C\subseteq\C_{D}$
is some set. Let $p\in S_{D}\left(MC\right)$. We say that $p$ is
\emph{tree-like (with respect to $M$,$C$)} if it is orthogonal to
every $q\in S_{D}^{<\omega}\left(MC\right)$ for which there exists
some $B\subseteq M$ with $|B|<|M|$ such that $q$ is $<\theta$
satisfiable in $B$.\end{defn}
\begin{prop}
\label{prop:tree-like}Let $M,C$ be as in Definition \ref{def:tree like}.
Suppose that $p\in S_{D}^{\alpha}\left(MC\right)$ is tree-like and
that $\left|C\right|<\kappa=\theta+\left|\alpha\right|^{+}$. Then
for every $B\subseteq M$ such that $\left|B\right|<\left|M\right|$
there exists some $E\subseteq M$ with $\left|E\right|<\kappa$ such
that $p|_{CE}\vdash p|_{CB}$. \end{prop}
\begin{proof}
It is enough to show that for any formula $\varphi\left(\bar{x},\bar{y},\bar{c}\right)$
where $\bar{c}$ is a finite tuple from $C$, there is some $E_{\varphi}\subseteq M$
such that $\left|E_{\varphi}\right|<\theta$ and 
\[
p|_{E_{\varphi}C}\vdash\left(p\upharpoonright\varphi\right)|_{B}=\set{\varphi\left(\bar{x},\bar{b},\bar{c}\right)\in p}{\bar{b}\in B^{\lg\left(\bar{y}\right)}}
\]
 (because then we let $E=\bigcup_{\varphi}E_{\varphi}$). 

Suppose not. Let $I=\left[M\right]{}^{<\theta}$ (all subsets of $M$
of size $<\theta$), then for every $E\in I$ there exists some $\bar{d}_{1}^{E},\bar{d}_{2}^{E}\in\mathfrak{C}_{D}^{\alpha}$,
$\bar{b}_{E}\in B^{\lg\left(\bar{y}\right)}$ such that $\bar{d}_{1}^{E},\bar{d}_{2}^{E}$
realize $p|_{EC}$ and $\mathfrak{C}_{D}\models\varphi\left(\bar{d}_{1}^{E},\bar{b}_{E},\bar{c}\right)\wedge\neg\varphi\left(\bar{d}_{2}^{E},\bar{b}_{E},\bar{c}\right)$.
By strong compactness, there is some $\theta$-complete ultrafilter
$\cal U$ on $I$ such that for every $X\in I$ we have $\set{Y\in I}{X\subseteq Y}\in\cal U$. 

By Lemma \ref{lem:average type along ultrafilter}, $r=\average{\sequence{\bar{d}_{1}^{E}\bar{d}_{2}^{E}\bar{b}_{E}}{E\in I}}{MC}{\cal U}\in S_{D}\left(MC\right)$.
Let $\bar{d}_{1},\bar{d}_{2}\in\mathfrak{C}_{D}^{\alpha}$ and $\bar{b}\in\mathfrak{C}_{D}^{<\omega}$
be such that $\bar{d}_{1}\bar{d}_{2}\bar{b}$ is a realization of
$r$. Now, $r'=\tp(\bar{b}/MC)$ is $<\theta$ satisfiable in $B$,
$\bar{d}_{1},\bar{d}_{2}$ realize $p$ (by our choice of $\cal U$)
but $\tp\left(\bar{d}_{1}/\bar{b}\bar{c}\right)\neq\tp\left(\bar{d}_{2}/\bar{b}\bar{c}\right)$
(as witnessed by $\varphi$). Hence $p$ is not orthogonal to $r'$,
which is a contradiction. \end{proof}
\begin{rem}
\label{rem:extension exists} Let $A\subseteq B\subseteq C\subseteq\C_{D}$.
If $p\in S_{D}^{n}\left(B\right)$ is $<\theta$ satisfiable in $A$
and $n<\omega$ then there is an extension $p\subseteq q\in S_{D}^{n}\left(C\right)$
which is $<\theta$ satisfiable in $A$. Indeed, let $\cal U_{0}=\set{\varphi\left(A^{n}\right)}{\varphi\in p}$,
note that it is $\theta$-complete, and extend it to a $\theta$-complete
ultrafilter $\cal U$ on all subsets of $A^{n}$. Let $q=\set{\varphi\left(\bar{x},\bar{c}\right)}{\bar{c}\subseteq C,\varphi\left(A^{n},\bar{c}\right)\in\cal U}$.
Now, as $\left|T\right|<\theta$ this type is a $D$-type: for any
finite tuple $\bar{c}$ from $C$, $q|_{\bar{c}}$ is realized by
some tuple from $A^{n}$ (as in the proof of Lemma \ref{def:avgtp}).
\end{rem}
\begin{thm}
\label{thm:tree} Let $M\prec\C_{D}$, $\bar{d}\in\mathfrak{C}_{D}^{<\lambda}$,
and $\bar{c}'\in\C_{D}^{<\lambda}$ be of length $<\kappa=\left|\lg\left(\bar{d}\right)\right|^{+}+\theta$.
Let $C=\bigcup\bar{c}'$ and suppose that $\aa\in\mxK$ and that $\tp\left(\bar{c}'/M\right)$
does not split over $B_{\aa}$. Then for any $r$, $\xx=\left(M,B_{\aa},\bar{d},\bar{c}'\bar{c}_{\aa},r\right)$
is a $\lambda$-tree-type decomposition (see Definition \ref{def:tree type}).\end{thm}
\begin{proof}
As $\tp\left(\bar{c}'/M\right)$ does not split over $B_{\aa}$, and
$\tp\left(\bar{c}_{\aa}/MC\right)$ does not split over $B_{\aa}$,
it follows that $\tp\left(\bar{c}'\bar{c}_{\aa}/M\right)$ does not
split over $B_{\aa}$. Let $\bar{c}=\bar{c}'\bar{c}_{\aa}$. We are
left to check that for every $A\subseteq M$ such that $\left|A\right|<\lambda$
there exists some $\bar{e}_{A}\in M^{<\kappa}$ where $\kappa=\left|\lg\left(\bar{d}\right)\right|^{+}+\theta$,
such that $\tp\left(\bar{d}/\bar{e}_{A}+\bar{c}\right)\vdash\tp\left(\bar{d}/A+\bar{c}\right)$.

By Proposition \ref{prop:tree-like} it is enough to prove that $p\left(\bar{x}\right)=\tp(\bar{d}/M+\bar{c})$
is tree-like (with respect to $M$, $\bar{c}$). Let $q\left(\bar{y}\right)\in S_{D}^{<\omega}\left(M+\bar{c}\right)$
be some type which is $<\theta$ satisfiable in some $B\subseteq M$
with $\left|B\right|<\lambda$. Suppose that $p$ is not orthogonal
to $q$. This means that there are $\bar{d}_{1},\bar{d}_{2},\bar{b}_{1},\bar{b}_{2}$
in $\C_{D}$ such that $\bar{d}_{1},\bar{d}_{2}\models p$, $\bar{b}_{1},\bar{b}_{2}\models q$
and $\bar{d}_{1}\bar{b}_{1}\not\equiv_{M\bar{c}}\bar{d}_{2}\bar{b}_{2}$.
By homogeneity, we may assume $\bar{d}_{1}=\bar{d}_{2}=\bar{d}$.
Let $q'\left(\bar{y}\right)\in S_{D}^{<\omega}\left(M+\bar{c}\bar{b}_{1}\bar{b}_{2}\right)$
be an extension of $q$ which is $<\theta$ satisfiable in $B$ (which
exists by Remark \ref{rem:extension exists}), and let $\bar{b}\models q'$.
Then for some $i=1,2$, it must be that $\bar{d}\bar{b}_{i}\not\equiv_{M\bar{c}}\bar{d}\bar{b}$.
Let $\bb\geq\aa$ be $\left(B_{\aa}\cup B,\bar{c}_{\aa}\concat\left(\bar{b}_{i},\bar{b}\right)\right)$,
then easily $\bb\in\K$, which contradicts the maximality of $\aa$. 
\end{proof}
By Theorems \ref{thm:Maximum exists} and \ref{thm:tree}, we get
that:
\begin{cor}
\label{cor:existence of tree type dec}Suppose $\xx$ is a $\lambda$-tree-type
decomposition, and $\bar{d}_{0}\in\C_{D}^{<\lambda}$. Then there
exists some $\lambda$-tree-type decomposition $\yy\geq\xx$ such
that $\bar{d}_{\xx}\bar{d}_{0}=\bar{d}_{\yy}$ and $r_{\yy}=r_{\xx}$. \end{cor}
\begin{proof}
Apply Theorem \ref{thm:Maximum exists} with $\bar{d}=\bar{d}_{\xx}\bar{d}_{0}$,
$C=\bigcup\bar{c}_{\xx}$, $M=M_{\xx}$ and $\bb=\left(B_{\xx},\emptyset\right)$,
to get some $\bb\leq\aa\in\mxK$. Now apply Theorem \ref{thm:tree}
with $\bar{c}'=\bar{c}_{\xx}$, $\aa$ and $r_{\xx}$.
\end{proof}

\subsection{\label{sub:Self-solvable-decomposition}Self-solvable decomposition }

\begin{defn}
\label{def:self-solvable}Let $M\prec\mathfrak{C}_{D}$ be a $D$-model
of cardinality $\lambda$. A \emph{$\lambda$-self-solvable decomposition}%
\footnote{In \cite[Definition 3.6]{Sh950}, this is called $\mathbf{tK}$. %
} is a $\lambda$-tree-type decomposition $\left(M,B,\bar{d},\bar{c},r\right)$
such that for every $A\subseteq M$ with $\left|A\right|<\lambda$
there exists some $\bar{c}_{A}\bar{d}_{A}\in M^{<\lambda}$ with the
following properties:
\begin{enumerate}
\item \label{enu:same length}The tuple $\bar{c}_{A}$ has the same length
as $\bar{c}$ (so $<\kappa=\left|\lg\left(\bar{d}\right)\right|^{+}+\theta$)
and $\bar{d}_{A}$ has the same length as $\bar{d}$. 
\item \label{enu:realize r}$\left(\bar{c}_{\xx},\bar{d}_{\xx},\bar{c}_{A},\bar{d}_{A}\right)$
realize $r_{\xx}\left(\bar{x}_{\bar{c}_{\xx}},\bar{x}_{\bar{d}_{\xx}},\bar{x}_{\bar{c}_{\xx}}',\bar{x}_{\bar{d}_{\xx}}'\right)$.
\item \label{enu:same type over small set}$\left(\bar{c}_{A},\bar{d}_{A}\right)$
realize $\tp\left(\bar{c}_{\xx}\bar{d}_{\xx}/A\right)$. 
\item \label{enu:The-strong tree}The main point is that we extend point
(2) from Definition \ref{def:tree type} by demanding that $\tp\left(\bar{d}_{\xx}/\bar{c}_{A}+\bar{d}_{A}+\bar{c}_{\xx}\right)\vdash\tp\left(\bar{d}_{\xx}/A+\bar{c}_{\xx}+\bar{c}_{A}+\bar{d}_{A}\right)$. 
\end{enumerate}
\end{defn}
The first thing we would like to show is that under the assumption
that $\lambda$ is measurable, a $\lambda$-self-solvable decomposition
exists. In the first order case one can weaken the assumption to ask
that $\lambda$ is weakly compact (see \cite[Claim 3.27]{Sh950}).
However, we do not know how to extend this results to $D$-models,
so we omit it. 

Note that the trivial decomposition $\left(M,\emptyset,\emptyset,\emptyset,\emptyset\right)$
is a $\lambda$-self-solvable decomposition. 
\begin{prop}
\label{prop:finding self solvable by chasing tail} Let $M$ be a
$D$-saturated model of cardinality $\lambda$, with $\lambda>\theta$
measurable. Let $\cal U$ be a normal non-principal $\lambda$-complete
ultrafilter on $\lambda$. Let $\xx$ be a $\lambda$-self-solvable
decomposition with $M_{\xx}=M$, and let $\bar{d}\in\C_{D}^{<\lambda}$.
Also write $M$ as an increasing continuous union $\bigcup_{\alpha<\lambda}M_{\alpha}$
where $M_{\alpha}\subseteq M$ is of size $<\lambda$. Finally, let
$\kappa=\left|\lg\left(\bar{d}_{\xx}\bar{d}\right)\right|^{+}+\theta$. 

Then for any $n<\omega$, there is a set $U_{n}\in\cal U$ , a sequence
$\sequence{\left(\bar{c}_{\alpha,n},\bar{d}_{\alpha,n}\right)}{\alpha\in U_{n}\cup\left\{ \lambda\right\} }$,
a type $r_{n}$ and a set $B_{n}\subseteq M$ with $\left|B_{n}\right|<\lambda$
such that the following holds:
\begin{enumerate}
\item For each $n<\omega$, $U_{n+1}\subseteq U_{n}$, $\xx_{n}=\left(M,B_{n},\bar{c}_{\lambda,n},\bar{d}_{\lambda,n},r_{n}\right)$
is a $\lambda$-tree-type decomposition, $\xx\leq\xx_{n}\leq\xx_{n+1}$
and $\bar{d}_{\xx}\bar{d}\trianglelefteq\bar{d}_{\lambda,n}$. Also,
$\lg\left(\bar{d}_{\lambda,n}\right),\lg\left(\bar{c}_{\lambda,n}\right)<\kappa$. 
\item For each $n<\omega$ and $\alpha\in U_{n}\cup\left\{ \lambda\right\} $,
$\bar{c}_{\alpha,n-1},\bar{d}_{\alpha,n-1}\init\bar{c}_{\alpha,n},\bar{d}_{\alpha,n}$,
and when $\alpha<\lambda$ they are in $M$, $\left(\bar{c}_{\lambda,n},\bar{d}_{\lambda,n},\bar{c}_{\alpha,n},\bar{d}_{\alpha,n}\right)\models r_{n}$
(so $r_{n}$ is increasing) and $\bar{c}_{\alpha,n}\bar{d}_{\alpha,n}$
realizes $\tp\left(\bar{c}_{\lambda,n}\bar{d}_{\lambda,n}/M_{\alpha}\right)$
(where $\bar{c}_{\alpha,-1},\bar{d}_{\alpha,-1}=\emptyset$). 
\item For each $n<\omega$ and $\alpha\in U_{n}$, $\tp\left(\bar{c}_{\lambda,n},\bar{d}_{\lambda,n},\bar{c}_{\alpha,n},\bar{d}_{\alpha,n}\right)$
contains $r_{\xx}$ (when restricted to the appropriate variables). 
\item For each $n<\omega$ and $\alpha\in U_{n}$, 
\[
\tp\left(\bar{d}_{\lambda,n}/\bar{c}_{\lambda,n}+\bar{d}_{\alpha,n+1}\right)\vdash\tp\left(\bar{d}_{\lambda,n}/\bar{c}_{\lambda,n}+\bar{c}_{\alpha,n}+\bar{d}_{\alpha,n}+M_{\alpha}\right).
\]

\end{enumerate}
\end{prop}
\begin{proof}
The construction is by induction on $n$. 

Assume $n=0$. Let $\bar{d}_{\lambda,0}=\bar{d}_{\xx}\bar{d}$ and
let $\bar{c}_{\lambda,0}\in\C_{D}^{<\kappa}$, $B_{0}$ be such that
$\xx\leq\left(M,B_{0},\bar{d}_{\lambda,0},\bar{c}_{\lambda,0},r_{\xx}\right)$
is a $\lambda$-tree-type decomposition (which exists by Corollary
\ref{cor:existence of tree type dec}). For $\alpha<\lambda$, as
$\xx$ is a self-solvable decomposition, there are $\bar{c}_{\alpha,\xx},\bar{d}_{\alpha,\xx}$
in $M$ which realize $\tp\left(\bar{c}_{\xx},\bar{d}_{\xx}/M_{\alpha}\right)$
such that $\left(\bar{c}_{\xx},\bar{d}_{\xx},\bar{c}_{\alpha,\xx},\bar{d}_{\alpha,\xx}\right)\models r_{\xx}$. 

Go on to find $\bar{c}_{\alpha,\xx},\bar{d}_{\alpha,\xx}\init\bar{c}_{\alpha,0},\bar{d}_{\alpha,0}$
in $M$ which realize $\tp\left(\bar{c}_{\lambda,0}\bar{d}_{\lambda,0}/M_{\alpha}\right)$
(exists as $M$ is $D$-saturated). By Corollary \ref{cor:full-indiscernibles},
we can find $U_{0}$ such that $\sequence{\bar{c}_{\alpha,0}\bar{d}_{\alpha,0}}{\alpha\in U_{0}}$
is a fully indiscernible sequence over $\bar{c}_{\lambda,0}+\bar{d}_{\lambda,0}$.
Let $r_{0}=\tp\left(\bar{c}_{\lambda,0},\bar{d}_{\lambda,0},\bar{c}_{\alpha,0},\bar{d}_{\alpha,0}/\emptyset\right)$,
where $\alpha\in U_{0}$. 

Assume $n=m+1$. Note that $\kappa=\left|\lg\left(\bar{d}_{\lambda,m}\right)\right|^{+}+\theta$.
For $\alpha\in U_{m}$, let $\bar{e}_{\alpha,m}\in M^{<\kappa}$ be
such that $\tp\left(\bar{d}_{\lambda,m}/\bar{c}_{\lambda,m}+\bar{e}_{\alpha,m}\right)\vdash\tp\left(\bar{d}_{\lambda,m}/\bar{c}_{\lambda,m}+\bar{d}_{\alpha,m}+\bar{c}_{\alpha,m}+M_{\alpha}\right)$,
which exists as $\xx_{m}$ is a tree-type decomposition. As $\kappa<\lambda$,
by restricting $U_{m}$, we may assume that $\bar{e}_{\alpha,m}$
has a constant length, independent of $\alpha$. Further, let us assume
that $\sequence{\bar{d}_{\alpha,m}\bar{c}_{\alpha,m}\bar{e}_{\alpha,m}}{\alpha\in U_{m}}$
is fully indiscernible. Let $\bar{e}_{\lambda,m}$ be such that $\bar{d}_{\lambda,m}\bar{c}_{\lambda,m}\bar{e}_{\lambda,m}\models\bigcup\set{\tp\left(\bar{d}_{\alpha,m}\bar{c}_{\alpha,m}\bar{e}_{\alpha,m}/M_{\alpha}\right)}{\alpha\in U_{m}}$.
This is a type by full indiscernibility, and such a tuple can be found
in $\C_{D}$ since $\bar{d}_{\lambda,m}\bar{c}_{\lambda,m}$ already
realize this union when we restrict to the appropriate variables,
by point (2). 

Now we essentially repeat the case $n=0$, applying Corollary \ref{cor:existence of tree type dec}
with $\bar{d}_{0},\xx$ there being $\bar{e}_{\lambda,m},\xx_{m}$
to find $B_{n}$ and $\bar{c}_{\lambda,n}$, but now we want that
$\bar{c}_{\alpha,m}\init\bar{c}_{\alpha,n}$ and $\bar{d}_{\alpha,m}\bar{e}_{\alpha,m}=\bar{d}_{\alpha,n}$
for $\alpha\in U_{m}\cup\left\{ \lambda\right\} $, so we find these
tuples and find $U_{n}$ such that $\sequence{\bar{c}_{\alpha,n}\bar{d}_{\alpha,n}}{\alpha\in U_{n}}$
is fully indiscernible over $\bar{c}_{\lambda,n}\bar{d}_{\lambda,n}$
and we let $r_{n}=\tp\left(\bar{c}_{\lambda,m},\bar{d}_{\lambda,m},\bar{c}_{\alpha,n},\bar{d}_{\alpha,n}/\emptyset\right)$. 

(In fact, in the proof we did not  need full indiscernibility at any
stage. In the case $n=0$ and the last stage of the successor step,
we only needed that $\tp\left(\bar{c}_{\lambda,m},\bar{d}_{\lambda,m},\bar{c}_{\alpha,n},\bar{d}_{\alpha,n}/\emptyset\right)$
is constant, and in the construction of the $\bar{e}_{\lambda,m}$
we only needed that the types $\tp\left(\bar{d}_{\alpha,m}\bar{c}_{\alpha,m}\bar{e}_{\alpha,m}/M_{\alpha}\right)$
are increasing with $\alpha$.)
\end{proof}

\begin{cor}
\label{cor:existsence of self-solvable}Let $M$ be a $D$-saturated
model of cardinality $\lambda$, where $\lambda\geq\theta$ is measurable,
and let $\bar{d}\in\C_{D}^{<\lambda}$. Let $\xx$ be some $\lambda$-self-solvable
decomposition, possibly trivial. Then there exists some $\lambda$-self-solvable
decomposition $\xx\le\yy$ such that $\bar{d}_{\xx}\bar{d}\init\bar{d}_{\yy}$. \end{cor}
\begin{proof}
Write $M=\bigcup_{\alpha<\lambda}M_{\alpha}$ where $M_{\alpha}\subseteq M$
are of cardinality $<\lambda$ and the sequence is increasing and
continuous. Also choose some normal ultrafilter $\cal U$ on $\lambda$.
Now we apply Proposition \ref{prop:finding self solvable by chasing tail},
to find $U_{n}$, $B_{n}$, $r_{n}$ and $\sequence{\left(\bar{c}_{\alpha,n},\bar{d}_{\alpha,n}\right)}{\alpha\in U_{n}\cup\left\{ \lambda\right\} }$.
Let $\bar{d}_{\lambda}=\bigcup_{n<\omega}\bar{d}_{\lambda,n}$, $\bar{c}_{\lambda}=\bigcup_{n<\omega}\bar{c}_{\lambda,n}$,
$B=\bigcup_{n<\omega}B_{n}$ and $r=\bigcup_{n<\omega}r_{n}$ (note
that this is indeed a $D$-type). Also, let $U=\bigcap_{n<\omega}U_{n}\in\cal U$
(as $\cal U$ is $\lambda$-complete). 

Then $\left(M,B,\bar{d}_{\lambda},\bar{c}_{\lambda},r\right)$ is
a $\lambda$-self-solvable decomposition: first of all it is a tree-type
decomposition, as $\tp\left(\bar{c}_{\lambda}/M\right)$ does not
split over $B$. Also, $\kappa=\left|\lg\left(\bar{d}_{\xx}\bar{d}\right)\right|^{+}+\theta$
is regular of cofinality $>\aleph_{0}$, so $\lg\left(\bar{c}_{\lambda}\right)<\kappa=\left|\lg\left(\bar{d}_{\lambda}\right)\right|^{+}+\theta$.
For each $A\subseteq M$ of size $<\lambda$, there is some $\alpha\in U$
such that $M_{\alpha}$ contains $A$. Let $\bar{c}_{A},\bar{d_{A}}=\bigcup_{n<\omega}\bar{c}_{\alpha,n},\bigcup_{n<\omega}\bar{d}_{\alpha,n}$.
Then, it follows from point (2) in Proposition \ref{prop:finding self solvable by chasing tail}
that $\left(\bar{c}_{\lambda},\bar{d}_{\lambda},\bar{c}_{A},\bar{d}_{A}\right)\models r$
and that $\left(\bar{c}_{A}\bar{d}_{A}\right)$ realize $\tp\left(\bar{c}_{\lambda}\bar{d}_{\lambda}/A\right)$.
Also, note that $r_{\xx}\subseteq r$, $B_{\xx}\subseteq B$, $\bar{c}_{\xx},\bar{d}_{\xx}\init\bar{c}_{\lambda},\bar{d}_{\lambda}$. 

Finally, we must check that $\tp\left(\bar{d}_{\lambda}/\bar{c}_{A}+\bar{d}_{A}+\bar{c}_{\lambda}\right)\vdash\tp\left(\bar{d}_{\lambda}/A+\bar{c}_{\lambda}+\bar{c}_{A}+\bar{d}_{A}\right)$.
This holds since formulas have finitely many variables. 
\end{proof}

\section{\label{sec:Finding-a-good}Finding a good family}

In this section we will show that the family of $\lambda$-self-solvable
decompositions is a good family of $\lambda$-decompositions whenever
$\lambda>\theta$ is measurable (note that in that case $\lambda^{<\lambda}=\lambda$).
This will conclude the proof of Conjecture \ref{conj:The generic pair conjecture}
in this case. So let $\Ff$ be the family of $\lambda$-self-solvable
decompositions $\xx$ such that $M_{\xx}$ is $D$-saturated of cardinality
$\lambda$. Let us go over Definition \ref{def:good family}, and
prove that each clause is satisfied by $\Ff$. 
\begin{claim}
Points (\ref{enu:invariant under isom}), (\ref{enu:every model is saturated}),
(\ref{enu:non-empty}), (\ref{enu:enlarging}), (\ref{enu:enbase})
and (\ref{enu:iso-extension}) are satisfied by $\Ff$.\end{claim}
\begin{proof}
Everything is clear, except (\ref{enu:enlarging}), which is exactly
Corollary \ref{cor:existsence of self-solvable}.
\end{proof}
We now move on to point (\ref{enu:union}), but for this we will need
the following lemma. 
\begin{lem}
\label{lem:finding indiscernibles} Suppose that $\left(I,<\right)$
is some linearly ordered set. Let $\sequence{\bar{a}_{i}}{i\in I}$
be a sequence of tuples of the same length from $\C_{D}$, and let
$B\subseteq\C_{D}$ be some set. Assume the following conditions.

\begin{enumerate}
\item For all $i\in I$, $\bar{a}_{i}=\bar{c}_{i}\bar{d}_{i}$.
\item For all $i\in I$, $\tp\left(\bar{a}_{i}/B_{i}\right)$ is increasing
with $i$, where \textbf{$B_{i}=B\cup\set{\bar{a}_{j}}{j<i}$.}
\item For all $i\in I$, $\tp\left(\bar{c}_{i}/B_{i}\right)$ does not split
over $B$.
\item For every $j<i$ in $I$, $\tp\left(\bar{d}_{i}/\bar{c}_{i}+\bar{a}_{j}\right)\vdash\tp\left(\bar{d}_{i}/\bar{c}_{i}+\bar{a}_{j}+B_{j}\right)$. 
\item For every $i_{1}<i_{2}$, $j_{1}<j_{2}$ from $I$, $\tp\left(\bar{a}_{i_{2}}\bar{a}_{i_{1}}/\emptyset\right)=\tp\left(\bar{a}_{j_{2}}\bar{a}_{j_{1}}/\emptyset\right)$.
\end{enumerate}
Then $\sequence{\bar{a}_{i}}{i\in I}$ is indiscernible over $B$.\end{lem}
\begin{proof}
We prove by induction on $n$ that $\sequence{\bar{a}_{i}}{i\in I}$
is an $n$-indiscernible sequence over $B$. 

For $n=1$ it follows from (2).

Now suppose that $\sequence{\bar{a}_{i}}{i\in I}$ is $n$-indiscernible
over $B$. Let $i_{1}<\ldots<i_{n}<i_{n+1}\in I$ and $j_{1}<\ldots<j_{n}<j_{n+1}\in I$
be such that, without loss of generality, $i_{n+1}\leq j_{n+1}$.
By (2), we know that $\bar{a}_{i_{1}}\ldots\bar{a}_{i_{n}}\bar{a}_{i_{n+1}}\equiv_{B}\bar{a}_{i_{1}}\ldots\bar{a}_{i_{n}}\bar{a}_{j_{n+1}}$.
By (3) and the induction hypothesis, we know that $\bar{a}_{i_{1}}\ldots\bar{a}_{i_{n}}\bar{c}_{j_{n+1}}\equiv_{B}\bar{a}_{j_{1}}\ldots\bar{a}_{j_{n}}\bar{c}_{j_{n+1}}$.
Combining, we get that $\bar{a}_{i_{1}}\ldots\bar{a}_{i_{n}}\bar{c}_{i_{n+1}}\equiv_{B}\bar{a}_{j_{1}}\ldots\bar{a}_{j_{n}}\bar{c}_{j_{n+1}}$. 

Suppose that $\varphi\left(\bar{d}_{i_{n+1}},\bar{c}_{i_{n+1}},\bar{a}_{i_{n}},\ldots\bar{a}_{i_{1}},\bar{b}\right)$
holds where $\bar{b}$ is a finite tuple from $B$. Let $r\left(\bar{x}_{\bar{d}},\bar{x}_{\bar{c}},\bar{x}_{\bar{a}}\right)=\tp\left(\bar{d}_{i_{n+1}},\bar{c}_{i_{n+1}},\bar{a}_{i_{n}}/\emptyset\right)$.
By (4), $r\left(\bar{x}_{\bar{d}},\bar{c}_{i_{n+1}},\bar{a}_{i_{n}}\right)\vdash\varphi\left(\bar{x}_{\bar{d}},\bar{c}_{i_{n+1}},\bar{a}_{i_{n}},\ldots\bar{a}_{i_{1}},\bar{b}\right)$.
Applying the last equation, we get that $r\left(\bar{x}_{\bar{d}},\bar{c}_{j_{n+1}},\bar{a}_{j_{n}}\right)\vdash\varphi\left(\bar{x}_{\bar{d}},\bar{c}_{j_{n+1}},\bar{a}_{j_{n}},\ldots\bar{a}_{j_{1}},\bar{b}\right).$
By (5), $r=\tp\left(\bar{d}_{j_{n+1}},\bar{c}_{j_{n+1}},\bar{a}_{j_{n}}/\emptyset\right)$,
so $\bar{d}_{j_{n+1}}$ satisfies the left hand side, and so also
the right hand side, and so $\varphi\left(\bar{d}_{j_{n+1}},\bar{c}_{j_{n+1}},\bar{a}_{j_{n}},\ldots\bar{a}_{j_{1}},\bar{b}\right)$
holds and we are done. \end{proof}
\begin{cor}
\label{cor:finding indiscernibles}Suppose that $\xx\in\Ff$, and
let $M=M_{\xx}$. Let $B\supseteq B_{\xx}$ be any subset of $M$
of cardinality $<\lambda$, and let $\alpha\leq\lambda$. For $i<\alpha$,
let $\bar{a}_{i}$ be such that $\bar{a}_{0}=\bar{c}_{B}\bar{d}_{B}$
(see Definition \ref{def:self-solvable}), and for $i>0$, $\bar{a}_{i}=\bar{c}_{B_{i}}\bar{d}_{B_{i}}$
where $B_{i}=B\cup\set{\bar{a}_{j}}{j<i}$. Then $\sequence{\bar{a}_{i}}{i<\alpha}\concat\left\langle \bar{c}_{\xx}\bar{d}_{\xx}\right\rangle $
is an indiscernible sequence over $B$.\end{cor}
\begin{proof}
Apply Lemma \ref{lem:finding indiscernibles} with $I=\alpha+1$ (so
that $\bar{a}_{\alpha}=\bar{c}_{\xx}\bar{d}_{\xx}$). Let us check
that the conditions there hold. (1) is obvious. (2) holds as $\tp\left(\bar{a}_{i}/B_{i}\right)=\tp\left(\bar{c}_{\xx}\bar{d}_{\xx}/B_{i}\right)\supseteq\tp\left(\bar{c}_{\xx}\bar{d_{\xx}}/B_{j}\right)$
when $\alpha+1>i\geq j$. (3) holds as $\tp\left(\bar{c}_{\xx}/B_{i}\right)$
does not split over $B$, so the same is true for $\bar{c}_{i}$.
(5) holds because for $i_{1}<i_{2}<\alpha+1$, 
\begin{align*}
\tp\left(\bar{a}_{i_{2}}\bar{a}_{i_{1}}/\emptyset\right) & =\tp\left(\bar{c}_{\xx}\bar{d}_{\xx}\bar{c}_{i_{1}}\bar{d}_{i_{1}}/\emptyset\right)=r_{\xx}\\
 & =\tp\left(\bar{c}_{\xx}\bar{d}_{\xx}\bar{c}_{j_{1}}\bar{d}_{j_{1}}/\emptyset\right)=\tp\left(\bar{a}_{j_{2}}\bar{a}_{j_{1}}/\emptyset\right).
\end{align*}

Finally, (4) holds because $\tp\left(\bar{d}_{\xx}/\bar{c}_{\xx}+\bar{a}_{j}\right)\vdash\tp\left(\bar{d}_{\xx}/\bar{c}_{\xx}+\bar{a}_{j}+B_{j}\right)$,
and as $\bar{c}_{\xx}\bar{d}_{\xx}\equiv_{B_{j+1}}\bar{c}_{B_{i}}\bar{d}_{B_{i}}=\bar{a}_{i}$,
we can replace $\bar{c}_{\xx}\bar{d}_{\xx}$ by $\bar{c}_{B_{i}}\bar{d}_{B_{i}}$
in this implication by applying an automorphism of $\C_{D}$. \end{proof}
\begin{lem}
\label{lem:restriction is ok}Suppose that $\xx_{1}\leq\xx_{2}$ are
two $\lambda$-decompositions from $\Ff$. Then for every subset $A$
of $M$ of size $<\lambda$ containing $B_{\xx_{1}}$, and for any
choice of $\bar{c}_{A},\bar{d}_{A}$ which we get when we apply Definition
\ref{def:self-solvable} on $\xx_{2}$, their restrictions to $\lg\left(\bar{c}_{\xx_{1}}\right),\lg\left(\bar{d}_{\xx_{1}}\right)$
satisfy all the conditions in Definition \ref{def:self-solvable}.\end{lem}
\begin{proof}
Denote these restrictions by $\bar{c}_{A}',\bar{d}_{A}'$. As $r_{\xx_{1}}\subseteq r_{\xx_{2}}$,
we get Clause (\ref{enu:realize r}) of Definition \ref{def:self-solvable}
immediately. Clause (\ref{enu:same type over small set}) is also
clear, so we are left with (\ref{enu:The-strong tree}). Since $\xx_{1}\in\Ff$,
there are some $\bar{c}_{A}'',\bar{d}_{A}''$ in $M$ in the same
length as $\lg\left(\bar{c}_{\xx_{1}}\right),\lg\left(\bar{d}_{\xx_{1}}\right)$,
which we get when applying Definition \ref{def:self-solvable} on
$\xx_{1}$. It is enough to show that $\bar{d}_{\xx_{1}}\bar{c}_{\xx_{1}}\bar{c}_{A}''\bar{d}_{A}''\equiv_{A}\bar{d}_{\xx_{1}}\bar{c}_{\xx_{1}}\bar{c}_{A}'\bar{d}_{A}'$.
Note first that $\bar{c}_{A}''\bar{d}_{A}''\equiv_{A}\bar{c}_{A}'\bar{d}_{A}'$
by (\ref{enu:same type over small set}), and as $\tp\left(\bar{c}_{\xx_{1}}/M\right)$
does not split over $A$, we also get $\bar{c}_{\xx_{1}}\bar{c}_{A}''\bar{d}_{A}''\equiv_{A}\bar{c}_{\xx_{1}}\bar{c}_{A}'\bar{d}_{A}'$.
So suppose that $\mathfrak{C}_{D}\models\varphi\left(\bar{d}_{\xx_{1}},\bar{c}_{\xx_{1}},\bar{c}_{A}'',\bar{d}_{A}'',\bar{a}\right)$
where $\bar{a}$ is a finite tuple from $A$. By (\ref{enu:The-strong tree})
and (\ref{enu:realize r}), $r_{\xx_{1}}\left(\bar{c}_{\xx_{1}},\bar{x}_{\bar{d}_{\xx_{1}}},\bar{c}_{A}'',\bar{d}_{A}''\right)\vdash\varphi\left(\bar{x}_{\bar{d}_{\xx_{1}}},\bar{c}_{\xx_{1}},\bar{c}_{A}'',\bar{d}_{A}'',\bar{a}\right)$,
and applying the last equation, we get that $r_{\xx_{1}}\left(\bar{c}_{\xx_{1}},\bar{x}_{\bar{d}_{\xx_{1}}},\bar{c}_{A}',\bar{d}_{A}'\right)\vdash\varphi\left(\bar{x}_{\bar{d}_{\xx_{1}}},\bar{c}_{\xx_{1}},\bar{c}_{A}',\bar{d}_{A}',\bar{a}\right)$,
but as $\bar{d}_{\xx_{1}}$ satisfies the left hand side (because
$r_{\xx_{1}}\subseteq r_{\xx_{2}}$), we are done. \end{proof}
\begin{thm}
\label{thm:union} Suppose $\delta<\lambda$ is a limit ordinal. Let
$\sequence{\xx_{j}}{j<\delta}$ be an increasing sequence of decompositions
from $\Ff$. Then $\xx=\sup_{j<\delta}\xx_{j}\in\Ff$. Hence point
(\ref{enu:union}) of Definition \ref{def:good family} is satisfied
by $\Ff$. \end{thm}
\begin{proof}
Easily $\xx$ is a $\lambda$-decomposition (i.e., $\left|B_{\xx}\right|<\lambda$
and $r_{\xx}$ is well defined). Also, $\tp\left(\bar{c}_{\xx}/M\right)$
does not split over $B_{\xx}=\bigcup B_{\xx_{i}}$, where we let $M=M_{\xx}$. 

Let $A\subseteq M$ be of cardinality $<\lambda$ and without loss
of generality suppose $B_{\xx}\subseteq A$. 

In order to prove the theorem, we need to find some $\bar{c},\bar{d}\in M^{<\lambda}$
in the same length as $\bar{c}_{\xx},\bar{d}_{\xx}$ such that $\tp\left(\bar{c}\bar{d}/A\right)=\tp\left(\bar{c}_{\xx}\bar{d}_{\xx}/A\right)$,
$\tp\left(\bar{c}_{\xx},\bar{d}_{\xx},\bar{c},\bar{d}\right)=r_{\xx}$,
and $\tp\left(\bar{d}_{\xx}/\bar{c}_{\xx}+\bar{c}+\bar{d}\right)\vdash\tp\left(\bar{d}_{\xx}/\bar{c}_{\xx}+\bar{c}+\bar{d}+A\right)$. 

Let us simplify the notation by letting $\beta_{j}=\lg\left(\bar{c}_{\xx_{j}}\right),\gamma_{j}=\lg\left(\bar{d}_{\xx_{j}}\right)$.
Note that when $\beta_{j}$ and $\gamma_{j}$ are constant from some
point onwards, finding such $\bar{c},\bar{d}$ is done by just applying
Definition \ref{def:self-solvable} to some $\xx_{j}$, so although
the following argument works for this case as well, it is more interesting
when $\beta_{j}$ and $\gamma_{j}$ are increasing. 

For every $i<\delta$, let $\bar{c}_{i},\bar{d}_{i}=\bar{c}_{A_{i}}\bar{d}_{A_{i}}$
be as in Definition \ref{def:self-solvable} applied to $\xx_{i}$
(so their length is $\beta_{i},\gamma_{i}$), where $A_{i}=A\cup\set{\bar{c}_{j},\bar{d}_{j}}{j<i}$.
Now repeat this process starting with $A_{\delta}$ to construct $\bar{c}_{i},\bar{d}_{i}$
for $\delta\leq i<\delta+\delta$.

Now we repeat this process $\kappa+1$ times, for $\kappa=\mu\left(D\right)^{+}+\left|T\right|^{+}<\lambda$,
to construct $\bar{c}_{i},\bar{d}_{i}$ and $A_{i}$ for $\delta+\delta\leq i<\delta\cdot\kappa+\delta$.
For $j<\delta$, let $O_{j}\subseteq\delta\cdot\kappa+\delta$ be
the set of all ordinals $i$ such that $i\modp{\delta}\geq j$. By
Corollary \ref{cor:finding indiscernibles} and Lemma \ref{lem:restriction is ok},
for each $j<\delta$, the sequence $I_{j}=\sequence{\left(\bar{c}_{i}\upharpoonright\beta_{j},\bar{d}_{i}\upharpoonright\gamma_{j}\right)}{i\in O_{j}}\concat\left\langle \left(\bar{c}_{\xx_{j}}\bar{d}_{\xx_{j}}\right)\right\rangle $
is an indiscernible sequence over $A$. 

Let $O_{j}'=O_{j}\cap\delta\cdot\kappa$, $O_{j}''=O_{j}\cap\left[\delta\cdot\kappa,\delta\cdot\kappa+\delta\right)$,
and let $I_{j}'=I_{j}\upharpoonright O_{j}'$, $I_{j}''=I_{j}\upharpoonright O_{j}''$.
As $O_{j}'$ has cofinality $\kappa$ (suppose $X\subseteq O_{j}'$
is unbounded, then the set $\set{i<\kappa}{X\cap\left[\delta\cdot i,\delta\cdot i+\delta\right)\neq\emptyset}$
is unbounded, so has cardinality $\kappa$, so $\left|X\right|\geq\kappa$,
but easily, the set $\set{\delta\cdot i+j}{i<\kappa}$ is cofinal
in $O_{j}'$), we can apply Lemma \ref{lem:average type on indiscernible sequence},
and consider the type $q_{j}\left(\bar{x}\right)=\average{I_{j}'}{A_{\delta\cdot\kappa+\delta}}{}$,
which is a complete $D$-type. So each $q_{j}$ is a type in $\beta_{j}+\gamma_{j}$
variables. 
\begin{claim*}
For $j_{1}<j_{2}$, $q_{j_{1}}\subseteq q_{j_{2}}$.\end{claim*}
\begin{proof}
Suppose $\varphi\left(\bar{y},\bar{a}\right)\in q_{j_{1}}$, where
$\bar{a}$ is a finite tuple from $A_{\delta\cdot\kappa+\delta}$
and $\bar{y}$ is a finite subtuple of variables of $\bar{x}$. By
definition, it means that for large enough $i\in O_{j_{1}}'$, $\varphi\left(\bar{c}_{i}\upharpoonright\beta_{j_{1}},\bar{d}_{i}\upharpoonright\gamma_{j_{1}},\bar{a}\right)$
holds (where we restrict $\bar{c}_{i},\bar{d}_{i}$ to $\bar{y}$,
of course). But $j_{2}>j_{1}$, so $O_{j_{2}}'\subseteq O_{j_{1}}'$,
so the same is true for $O_{j_{2}}'$, and so $\varphi\left(\bar{y},\bar{a}\right)\in q_{j_{2}}$. 
\end{proof}
Let $q=\bigcup_{j<\delta}q_{j}$. As $\delta$ is limit, it follows
that $q$ is also a $D$-type over $A_{\delta\cdot\kappa+\delta}$.
Let $\bar{c}',\bar{d}'\models q$, and for each $j<\delta$, let $\bar{c}'_{j}=\bar{c}\upharpoonright\beta_{j},\bar{d}_{j}'=\bar{d}'\upharpoonright\gamma_{j}$.
It now follows that for each $j<\delta$, the sequence $I_{j}'\concat\left\langle \bar{c}_{j}',\bar{d}_{j}'\right\rangle \concat I_{j}''$
is indiscernible over $A$. 

Let us check that $\bar{c}',\bar{d}'$ are as required. To show this
it is enough to see that for every $j<\delta$, $\bar{c}_{\xx_{j}}\bar{d}_{\xx_{j}}\bar{c}_{j}'\bar{d}_{j}'\equiv_{A}\bar{c}_{\xx_{j}}\bar{d}_{\xx_{j}}\bar{c}_{j}\bar{d_{j}}$.
Suppose $\varphi\left(\bar{c}_{\xx_{j}},\bar{d}_{\xx_{j}},\bar{c}_{j},\bar{d}_{j},\bar{a}\right)$
holds, where $\bar{a}$ is a finite tuple from $A$. By indiscernibility,
$\varphi\left(\bar{c}_{\xx_{j}},\bar{d}_{\xx_{j}},\bar{c}_{\delta\cdot\kappa+j},\bar{d}_{\delta\cdot\kappa+j},\bar{a}\right)$
holds as well. By choice of $\bar{c}_{\delta\cdot\kappa+j+1},\bar{d}_{\delta\cdot\kappa+j+1}$,
it follows that 
\begin{equation}
r_{\xx_{j}}\left(\bar{c}_{\xx_{j}},\bar{x}_{\bar{d}_{\xx_{j}}},\bar{c}_{\delta\cdot\kappa+j+1}\upharpoonright\beta_{j},\bar{d}_{\delta\cdot\kappa+j+1}\upharpoonright\gamma_{j}\right)\vdash\varphi\left(\bar{c}_{\xx_{j}},\bar{x}_{\bar{d}_{\xx_{j}}},\bar{c}_{\delta\cdot\kappa+j},\bar{d}_{\delta\cdot\kappa+j},\bar{a}\right).\tag{*}\label{eq:implication}
\end{equation}
By indiscernibility, 
\begin{align*}
 & \left(\bar{c}_{\delta\cdot\kappa+j+1}\upharpoonright\beta_{j}\right)\left(\bar{d}_{\delta\cdot\kappa+j+1}\upharpoonright\gamma_{j}\right)\bar{c}_{\delta\cdot\kappa+j}\bar{d}_{\delta\cdot\kappa+j}\equiv_{A}\\
 & \left(\bar{c}_{\delta\cdot\kappa+j+1}\upharpoonright\beta_{j}\right)\left(\bar{d}_{\delta\cdot\kappa+j+1}\upharpoonright\gamma_{j}\right)\left(\bar{c}'\upharpoonright\beta_{j}\right)\left(\bar{d}'\upharpoonright\gamma_{j}\right),
\end{align*}
and as $\tp\left(\bar{c}_{\xx}/M\right)$ does not split over $A$,
\begin{align*}
 & \bar{c}_{\xx_{j}}\left(\bar{c}_{\delta\cdot\kappa+j+1}\upharpoonright\beta_{j}\right)\left(\bar{d}_{\delta\cdot\kappa+j+1}\upharpoonright\gamma_{j}\right)\bar{c}_{\delta\cdot\kappa+j}\bar{d}_{\delta\cdot\kappa+j}\equiv_{A}\\
 & \bar{c}_{\xx_{j}}\left(\bar{c}_{\delta\cdot\kappa+j+1}\upharpoonright\beta_{j}\right)\left(\bar{d}_{\delta\cdot\kappa+j+1}\upharpoonright\gamma_{j}\right)\bar{c}'_{j}\bar{d}'_{j}.
\end{align*}
Applying the last equation to (\ref{eq:implication}), we get that
\begin{equation}
r_{\xx_{j}}\left(\bar{c}_{\xx_{j}},\bar{x}_{\bar{d}_{\xx_{j}}},\bar{c}_{\delta\cdot\kappa+j+1}\upharpoonright\beta_{j},\bar{d}_{\delta\cdot\kappa+j+1}\upharpoonright\gamma_{j}\right)\vdash\varphi\left(\bar{c}_{\xx_{j}},\bar{x}_{\bar{d}_{\xx_{j}}},\bar{c}'_{j},\bar{d}'_{j},\bar{a}\right).\tag{**}\label{eq:second implication}
\end{equation}
As $\bar{d}_{\xx_{j}}$ satisfies the left hand side of (\ref{eq:second implication}),
it also satisfies the right side, and we are done. \end{proof}
\begin{rem}
The proof of Theorem \ref{thm:union} as above can be simplified in
the case where $D$ is trivial (i.e., the usual first order case).
There, we would not need to introduce $\kappa$ (i.e., we can choose
$\kappa=1$), and we would not have to use dependence (which we used
in applying Lemma \ref{lem:average type on indiscernible sequence}
which states that the average type of an indiscernible sequence exists
and is a $D$-type). To make the proof work, we only needed to find
$\bar{c}',\bar{d}'$ such that the sequence $I_{j}'\concat\left\langle \bar{c}'\upharpoonright\beta_{j},\bar{d}'\upharpoonright\gamma_{j}\right\rangle \concat I_{j}''$
is indiscernible over $A$, and this can easily done by compactness. 
\end{rem}
We now move on to points (\ref{enu:isohomo}) and (\ref{enu:count})
of Definition \ref{def:good family}. 

Suppose $\xx$ is a $\lambda$-tree-type decomposition. Let $L_{\bar{c}_{\xx}}$
be the set of formulas $\varphi\left(\bar{x}_{\bar{c}_{\xx}},\bar{y}\right)$
where $\bar{x}_{\bar{c}_{\xx}}$ is a tuple of variables in the length
of $\bar{c}_{\xx}$ (of course only finitely many of them appear in
$\varphi$). For $B\subseteq M_{\xx}$ over which $\tp\left(\bar{c}_{\xx}/M_{\xx}\right)$
does not split, define $\sch{\xx}B:L_{\bar{c}_{\xx}}\to\SS\left(S_{D}^{<\omega}\left(B\right)\right)$
by: 
\[
\sch{\xx}B\left(\varphi\left(\bar{x}_{\bar{c}_{\xx}},\bar{y}\right)\right)=\set{p\left(\bar{y}\right)\in S_{D}\left(B\right)}{\exists\bar{e}\in M_{\xx}^{\lg\left(\bar{y}\right)}\left(\bar{e}\models p\land\C_{D}\models\varphi\left(\bar{c}_{\xx},\bar{e}\right)\right)}.
\]
As $\tp\left(\bar{c}_{\xx}/M\right)$ does not split over $B$, we
can also replace $\exists$ with $\forall$ in the definition of $\sch{\xx}B$.
This implies that for $B'\supseteq B$ and $p\in S_{D}\left(B'\right)$,
\begin{equation}
p\in\sch{\xx}{B'}\left(\varphi\right)\Leftrightarrow p|_{B}\in\sch{\xx}B\left(\varphi\right).\tag{\ensuremath{\dagger}}\label{eq:restriction}
\end{equation}
Suppose that $\yy$ is another $\lambda$-tree-type decomposition.
When $h$ is an elementary map from $B_{\xx}$ to $B_{\yy}$, then
it induces a well defined map from $S_{D}\left(B_{\xx}\right)$ to
$S_{D}\left(B_{\yy}\right)$ which we will also call $h$. So if $\bar{c}_{\xx}$
has the same length as $\bar{c}_{\yy}$, it makes sense to ask that
$h\circ\sch{\xx}{B_{\xx}}=\sch{\yy}{B_{\yy}}$. When $r_{\xx}=r_{\yy}$,
a partial elementary map $h$ whose domain is $B_{\xx}\cup\bigcup\bar{c}_{\xx}\cup\bigcup\bar{d}_{\xx}$
which maps $\left(\bar{d}_{\xx},\bar{c}_{\xx},B_{\xx}\right)$ onto
$\left(\bar{d}_{\yy},\bar{c}_{\yy},B_{\yy}\right)$ and satisfies
$h\circ\sch{\xx}{B_{\xx}}=\sch{\yy}{B_{\yy}}$ is called a \emph{pseudo
isomorphism} between $\xx$ and $\yy$. 

Note that if $h$ is a pseudo isomorphism, then for any two tuples
$\bar{a}$, $\bar{b}$, from $M_{\xx}$, $M_{\yy}$ respectively,
if $h\upharpoonright B_{\xx}$ can be extend to witness that $B_{\xx}\bar{a}\equiv B_{\yy}\bar{b}$,
then $\bar{c}_{\xx}B_{\xx}\bar{a}\equiv\bar{c}_{\yy}B_{\yy}\bar{a}$. 
\begin{prop}
\label{prop:pseudo isomorphism implies weak isomorphism}Suppose $\xx,\yy\in\Ff$
are such that $r_{\xx}=r_{\yy}$, and suppose that $h:\xx\to\yy$
is a pseudo isomorphism. Then $h$ is a weak isomorphism, i.e., it
extends to an isomorphism $h^{+}:\xx\to\yy$. Conversely, if $h$
is a weak isomorphism, then it is a pseudo isomorphism. \end{prop}
\begin{proof}
We will do a back and forth argument. In each successor step we will
add an element to either $B_{\xx}$ or $B_{\yy}$ and increase $h$.
In doing so, the new $\xx$ and $\yy$'s will still remain in $\Ff$
(by point (\ref{enu:enbase}) of Definition \ref{def:good family}
which is easily true for $\Ff$). In addition, the increased $h$'s
will still be pseudo isomorphisms by (\ref{eq:restriction}). In order
to do this, it is enough to do a single step, so assume that $h:\xx\to\yy$
is a pseudo isomorphism, and $a\in M_{\xx}$. We want to find $b\in M_{\yy}$
such that $h\cup\left\{ \left(a,b\right)\right\} $ is a pseudo isomorphism
from $\xx'=\left(M_{\xx},B_{\xx}\cup\left\{ a\right\} ,\bar{d}_{\xx},\bar{c}_{\xx},r_{\xx}\right)$
to $\yy'=\left(M_{\yy},B_{\yy}\cup\left\{ b\right\} ,\bar{d}_{\yy},\bar{c}_{\yy},r_{\yy}\right)$. 

Let $A=B_{\xx}\cup\left\{ a\right\} $, and let $\bar{c}_{A}^{\xx},\bar{d}_{A}^{\xx}$
be as in Definition \ref{def:self-solvable} for $\xx$. Let $\bar{c}_{B_{\yy}}^{\yy},\bar{d}_{B_{\yy}}^{\yy}$
be the parallel tuples for $\yy$ and $B_{\yy}$. By (\ref{enu:same type over small set})
of Definition \ref{def:self-solvable}, $B_{\xx}\bar{c}_{A}^{\xx}\bar{d}_{A}^{\xx}\equiv B_{\yy}\bar{c}_{B_{\yy}}^{\yy}\bar{d}_{B_{\yy}}^{\yy}$,
as witnessed by expanding $h\upharpoonright B_{\xx}$ to $B_{\xx}\bar{c}_{A}^{\xx}\bar{d}_{A}^{\xx}$.
Hence as $M_{\yy}$ is $D$-saturated there is some $b\in M_{\yy}$
such that $B_{\xx}a\bar{c}_{A}^{\xx}\bar{d}_{A}^{\xx}\equiv B_{\yy}b\bar{c}_{B_{\yy}}^{\yy}\bar{d}_{B_{\yy}}^{\yy}$.
So we have found our $b$. 

As noted above, as $h$ is a pseudo isomorphism, we get that 
\begin{equation}
B_{\xx}a\bar{c}_{A}^{\xx}\bar{d}_{A}^{\xx}\bar{c}_{\xx}\equiv B_{\yy}b\bar{c}_{B_{\yy}}^{\yy}\bar{d}_{B_{\yy}}^{\yy}\bar{c}_{\yy}.\tag{\ensuremath{\dagger\dagger}}\label{eq:adding c}
\end{equation}
Suppose now that $\varphi\left(\bar{d}_{\xx},\bar{c}_{\xx},a,\bar{e}\right)$
holds, where $\bar{e}$ is a finite tuple from $B_{\xx}$. By the
choice of $\bar{c}_{A}^{\xx},\bar{d}_{A}^{\xx}$, $r_{\xx}\left(\bar{c}_{\xx},\bar{x}_{\bar{d}_{\xx}},\bar{c}_{A}^{\xx},\bar{d}_{A}^{\xx}\right)\vdash\varphi\left(\bar{x}_{\bar{d}_{\xx}},\bar{c}_{\xx},a,\bar{e}\right)$.
Applying (\ref{eq:adding c}), we get that $r_{\xx}\left(\bar{c}_{\yy},\bar{x}_{\bar{d}_{\yy}},\bar{c}_{B_{\yy}}^{\yy},\bar{d}_{B_{\yy}}^{\yy}\right)\vdash\varphi\left(\bar{x}_{\bar{d}_{\yy}},\bar{c}_{\yy},b,h\left(\bar{e}\right)\right)$.
As $r_{\xx}=r_{\yy}$, $\bar{d}_{\yy}$ realizes the left hand side,
so also the right hand side and so $\C_{D}\models\varphi\left(\bar{d}_{\yy},\bar{c}_{\yy},b,h\left(\bar{e}\right)\right)$.

For the limit stages, note that if $\sequence{h_{i}}{i<\delta}$ is
an increasing sequence of pseudo isomorphisms $h_{i}:\xx_{i}\to\yy_{i}$
where $\xx_{i}=\left(M_{\xx},B_{\xx_{i}},\bar{d}_{\xx},\bar{c}_{\xx},r_{\xx}\right)$
and $\yy_{i}=\left(M_{\yy},B_{\yy_{i}},\bar{d}_{\yy},\bar{c}_{\yy},r_{\yy}\right)$
are increasing, and $\delta<\lambda$, then $\bigcup\set{h_{i}}{i<\delta}$
is a pseudo isomorphism from $\sup_{i<\delta}\xx_{i}$ to $\sup_{i<\delta}\yy_{i}$.

The other direction is immediate. \end{proof}
\begin{cor}
Clause (\ref{enu:isohomo}) in Definition \ref{def:good family} holds
for $\Ff$. \end{cor}
\begin{proof}
We are given two increasing sequences of decompositions $\sequence{\xx_{i}}{i<\delta}$
and $\sequence{\yy_{i}}{i<\delta}$ in $\Ff$, and we assume that
for each $i<\delta$ there is a weak isomorphism $g_{i}:\xx_{i}\to\yy_{i}$
such that $g_{i}\subseteq g_{i}$ whenever $i<j$. We need to show
that the union $g=\bigcup_{i<\delta}g_{i}$ is also a weak isomorphism
from $\xx=\sup_{i<\delta}\xx_{i}$ to $\yy=\sup_{i<\delta}\yy_{i}$.
We already know by Theorem \ref{thm:union} that $\xx,\yy\in\Ff$,
so by Proposition \ref{prop:pseudo isomorphism implies weak isomorphism},
we only need to show that $g$ is a pseudo isomorphism and that $r_{\xx}=r_{\yy}$.
The latter is clear, as $r_{\xx}=\bigcup_{i<\delta}r_{\xx_{i}}=\bigcup_{i<\delta}r_{\yy_{i}}=r_{\yy}$.
Also, it is clear that $g$ is an elementary map taking $\left(\bar{d}_{\xx},\bar{c}_{\xx},B_{\xx}\right)$
to $\left(\bar{d}_{\yy},\bar{c}_{\yy},B_{\yy}\right)$. 

Note that $L_{\bar{c}_{\xx}}=\bigcup_{i<\delta}L_{\bar{c}_{\xx_{i}}}$
and that for $\varphi\in L_{\bar{c}_{\xx_{i}}}$, $\sch{\xx}{B_{\xx}}\left(\varphi\right)=\sch{\xx_{i}}{B_{\xx}}\left(\varphi\right)$.
The same is true for $\yy$. Hence, for such $i<\delta$, $\varphi$
and for any $p\in S_{D}\left(B_{\yy}\right)$, 
\begin{align*}
p\in g\left(\sch{\xx}{B_{\xx}}\left(\varphi\right)\right) & \Leftrightarrow p\in g\left(\sch{\xx_{i}}{B_{\xx}}\left(\varphi\right)\right)\\
 & \Leftrightarrow p|_{B_{\yy_{i}}}\in g\left(\sch{\xx_{i}}{B_{\xx_{i}}}\left(\varphi\right)\right)\\
 & \Leftrightarrow p|_{B_{\yy_{i}}}\in g_{i}\left(\sch{\xx_{i}}{B_{\xx_{i}}}\left(\varphi\right)\right)\\
 & \Leftrightarrow p|_{B_{\yy_{i}}}\in\sch{\yy_{i}}{B_{\yy_{i}}}\left(\varphi\right)\\
 & \Leftrightarrow p\in\sch{\yy_{i}}{B_{\yy}}\left(\varphi\right).
\end{align*}
\end{proof}
\begin{defn}
\label{def:externally def set}For a model $M\prec\C$ and $B\subseteq\C$,
we let $M_{\left[B\right]}$ be $M$ with predicates for all $B$-definable
subsets. More precisely, for each formula $\varphi\left(x_{1},\ldots,x_{n},\bar{b}\right)$
over $B$, we add a predicate $R_{\varphi\left(\bar{x},\bar{b}\right)}\left(\bar{x}\right)$
and we interpret it as $\varphi\left(\C^{n},\bar{b}\right)\cap M^{n}$.
If $B\subseteq M$, then this is definably equivalent to adding names
for elements of $B$. 
\end{defn}
For a $\lambda$-decomposition $\xx$, denote by $M_{\left[\xx\right]}$
the structure $M_{\left[\bar{c}_{\xx}+\bar{d}_{\xx}+B_{\xx}\right]}$. 
\begin{thm}
\label{thm:externally definable sets, homogeneous}Suppose $\xx\in\Ff$.
Then $M_{\left[\xx\right]}$ is homogeneous. \end{thm}
\begin{proof}
We have to show that if $A\subseteq M$ is of cardinality $<\lambda$,
and $f$ is a partial elementary map of $M_{\left[\xx\right]}$ with
domain $A$, then we can extend it to an automorphism. We may assume
that $B_{\xx}\subseteq A$ and that $f\upharpoonright B_{\xx}=\id$,
as $f$ preserves all $B_{\xx}$-definable sets. It follows that $\xx'=\left(M_{\xx},A,\bar{d}_{\xx},\bar{c}_{\xx},r_{\xx}\right)$
and $\xx''=\left(M_{\xx},f\left(A\right),\bar{d}_{\xx},\bar{c}_{\xx},r_{\xx}\right)$
are both in $\Ff$. By definition, $f$ extends to an elementary map
$f':\left(A,\bar{d}_{\xx},\bar{c}_{\xx}\right)\to\left(f\left(A\right),\bar{d}_{\xx},\bar{c}_{\xx}\right)$,
but moreover $f$ is a pseudo isomorphism. This follows easily by
(\ref{eq:restriction}) above. Hence we are done by Proposition \ref{prop:pseudo isomorphism implies weak isomorphism}. \end{proof}
\begin{cor}
\label{cor:counting up to isomorphism}Clause (\ref{enu:count}) in
Definition \ref{def:good family} holds for $\Ff$. \end{cor}
\begin{proof}
Suppose $\set{\xx_{i}}{i<\lambda^{+}}$ is a set of pairwise non-isomorphic
elements of $\Ff$ with $M_{\xx_{i}}=M$ for all $i$. We may assume
that for some $\beta,\gamma<\lambda$ and all $i<\lambda^{+}$, $\bar{c}_{\xx_{i}}$
is of length $\beta$ and $\bar{d}_{\xx_{i}}$ is of length $\gamma$.
We may also assume, as $\lambda^{<\lambda}=\lambda$, that $B_{\xx_{i}}=B$
for all $i<\lambda^{+}$. Let $L'$ be the common language of the
structures $M_{\left[\xx_{i}\right]}$ (which we may assume is constant
as it only depends on the length of $\bar{c}_{\xx_{i}}$, $\bar{d}_{\xx_{i}}$
and $B_{\xx_{i}}$). Let $D_{i}=D\left(M_{\left[\xx_{i}\right]}\right)$
in the language $L'$ (recall that $D\left(A\right)$ consists of
all types of finite tuples from $A$ over $\emptyset$). The language
$L'$ has size $<\lambda$, so the number of possible $D$'s is $\leq2^{2^{\left|L'\right|}}<\lambda$,
so we may assume that $D_{i}=D_{0}$ for all $i<\lambda$ (it follows
that $M_{\xx_{i}}\equiv M_{\xx_{0}}$). Finally, we are done by Lemma
\ref{lem:Grossberg}, Corollary \ref{cor:homo-isom} and Theorem \ref{thm:externally definable sets, homogeneous}. \end{proof}
\begin{rem}
One can also prove Corollary \ref{cor:counting up to isomorphism}
directly, showing that the number of $\lambda$-decompositions in
$\Ff$ up to pseudo isomorphism is $\leq\lambda$, and then use Proposition
\ref{prop:pseudo isomorphism implies weak isomorphism}.
\end{rem}
Finally, we have proved that $\Ff$ is a good family of $\lambda$-decompositions,
so by Theorem \ref{thm:generic pair with good family} we get:
\begin{cor}
\label{cor:main}Conjecture \ref{conj:The generic pair conjecture},
and the conclusion of Theorem \ref{thm:clever counting of types}
hold when $\lambda$ is measurable.\end{cor}
\begin{problem}
To what extent can we generalize \cite[Theorem 7.3]{Sh950} to dependent
finite diagrams? For instance, is the generic pair conjecture for
dependent finite diagrams also true when $\lambda$ is weakly compact? 
\end{problem}
\bibliographystyle{alpha}
\bibliography{common2}

\end{document}